\documentclass[12pt]{amsart}
\usepackage[myheadings]{fullpage}
\usepackage{fancyhdr}
\usepackage{amsmath,amssymb,amsbsy,amsthm,amsfonts,epsfig,verbatim,enumerate,comment}
\usepackage{graphicx,xcolor,color,wrapfig}
\usepackage{pxfonts,bm}
\usepackage{hyperref}
\usepackage{pbox}
\input{xy}
\xyoption{all}
\CompileMatrices

\newcommand{\ed}{{\rm d}}
\newcommand{\w}{{\mathchoice{\,{\scriptstyle\wedge}\,}{{\scriptstyle\wedge}}
      {{\scriptscriptstyle\wedge}}{{\scriptscriptstyle\wedge}}}}

\newcommand{\del}{{\partial}}

\newcommand{\dels}[1]{\partial_{\sigma_{#1}}}
\newcommand{\delsb}[1]{\partial_{\sigmab_{#1}}}

\newcommand{\liealgebra}[1]{{\mathfrak {#1}}}
\newcommand{\g}{\liealgebra{g}}

\newcommand{\sla}{\liealgebra{sl}}

\newcommand{\liegroup}[1]{{\operatorname{#1}}}

\newcommand{\SL}{\liegroup{SL}}

\newcommand{\C}{\mathbb C}

\newcommand{\PP}{\mathbb P}
\DeclareMathOperator{\ad}{ad}

\newcommand{\tr}{\rm tr}
\newcommand{\dd}[2]{\frac{\partial {#1}}{\partial {#2}}}
\newcommand{\bdot}{\overset{\bm .}}

\usepackage{tikz}

\newcommand{\mcd}{\mathcal D}

\newcommand{\mcf}{\mathcal F}

\newcommand{\mcl}{\mathcal L}

\newcommand{\mco}{\mathcal O}

\newcommand{\mcs}{\mathcal S}
\newcommand{\mct}{\mathcal T}

\newcommand{\mcv}{\mathcal V}

\newcommand{\im}{\textnormal{i}}

\newcommand{\tne}{\textnormal{e}}

\newcommand{\balpha}{\boldsymbol{\alpha}}
\newcommand{\bomega}{\boldsymbol{\omega}}
\newcommand{\btheta}{\boldsymbol{\theta}}

\newcommand{\bphi}{\boldsymbol{\phi}}
\newcommand{\bPhi}{\boldsymbol{\Phi}}

\newcommand{\bsigma}{\boldsymbol{\sigma}}
\newcommand{\bzeta}{\boldsymbol{\zeta}}

\newcommand{\bmcs}{\boldsymbol{\mcs}} 
\newcommand{\bmg}{\boldsymbol{\g}} 

\newcommand{\bba}{\tb{a}}
\newcommand{\bbb}{\tb{b}}
\newcommand{\bbc}{\tb{c}}

\newcommand{\bP}{\tb{P}}\newcommand{\bbp}{\tb{p}}
\newcommand{\bQ}{\tb{Q}}

\newcommand{\bbt}{\tb{t}}

\newcommand{\bV}{\tb{V}}

\newcommand{\bY}{\tb{Y}}
\newcommand{\bZ}{\tb{Z}}

\newcommand{\cbphi}{\check{\boldsymbol{\phi}}}
\newcommand{\cbmcs}{\check{\boldsymbol{\mcs}}}

\newcommand{\hbphi}{\hat{\boldsymbol{\phi}}}
\newcommand{\hbmcs}{\hat{\boldsymbol{\mcs}}}
\newcommand{\hbY}{\hat{\bY}}
\newcommand{\hbP}{\hat{\bP}}
\newcommand{\hbmg}{\hat{\bmg}}

\newcommand{\ol}{\overline}

\newcommand{\ab}{\ol{a}}

\newcommand{\hb}{\ol{h}}

\newcommand{\tba}{\ol{t}}

\newcommand{\Uba}{\ol{U}}

\newcommand{\sigmab}{\ol{\sigma}}

\newcommand{\xib}{\ol{\xi}}
\newcommand{\tbar}[1]{t_{\ol{#1}}}

\newcommand{\Fh}[1]{\hat{\mcf}^{(#1)}}

\newcommand{\be}{\begin{equation}}
\newcommand{\ee}{\end{equation}}
\newcommand{\benu}{\begin{enumerate}}
\newcommand{\enu}{\end{enumerate}}
\newcommand{\beit}{\begin{itemize}}
\newcommand{\enit}{\end{itemize}}

\newcommand{\bp}{\begin{pmatrix}}
\newcommand{\ep}{\end{pmatrix}}

\newcommand{\lra}{\longrightarrow}

\newcommand{\n}{\notag}
\newcommand{\noi}{\noindent}
\newcommand{\np}{\newpage}
\newcommand{\tn}{\textnormal}
\newcommand{\tb}{\textbf}

\newcommand{\one}{\vspace{1mm}}
\newcommand{\two}{\vspace{2mm}}

\newcommand{\tcr}{\textcolor{black}}
\newcommand{\ttt}{\tb{\texttt{t}}}
\newcommand{\ftmark}{\footnotemark}
\newcommand{\fttext}{\footnotetext}
\newcommand{\sub}{\subsection}
\newcommand{\subb}{\subsubsection}

\newcommand{\Vir}{\mcv\tn{ir}}

\newcommand{\conjt}[1]{\ol{#1}^t}

\makeatletter

\newcommand{\Rmnum}[1]{\expandafter\@slowromancap\romannumeral #1@}
\makeatother

\usepackage[mathscr]{eucal}
\usepackage{bbm}
\usepackage{oldgerm}

\theoremstyle{plain}
\newtheorem{thm}{Theorem}[section]
\newtheorem{lem}[thm]{Lemma}

\newtheorem{cor}[thm]{Corollary}

\newtheorem{prop}[thm]{Proposition}

\newtheorem{rem}[thm]{Remark}

\theoremstyle{definition}
\newtheorem{defn}{Definition}[section]

\begin{document}

%%%%%%%%%%%%%%%%%%%%%%%%%%%%%%%%%%%%%%%%%
\title[CMC hierarchy \Rmnum{2}]
{CMC hierarchy \Rmnum{2}: Non-commuting symmetries and \\affine Kac-Moody algebra}
\author{Joe S. Wang}
%\address{Seoul, South Korea}
\email{jswang12@gmail.com}
%53C43, 35A27
\date{\today}
%differential geometry, exterior differential system,  minimal Lagrangian surface, characteristic cohomology, Jacobi field, conservation law,  formal Killing field, recursion}  %Integrable extension, Non-local symmetry, Secondary characteristic cohomology, Affine algebra, Noether's theorem}
%\thanks{
%}
\begin{abstract}
We propose a further extension of the structure equation for a truncated  CMC hierarchy
by the non-commuting, truncated Virasoro algebra of non-local symmetries.
Via a canonical dressing transformation, we first define the wave function for a truncated CMC hierarchy.
This leads to a pair of additional formal Killing fields, 
and  the corresponding spectral Killing field is defined 
by a purely algebraic formula up to an integrable extension.
The extended CMC hierarchy is obtained by packaging these data
into the associated affine Kac-Moody algebra valued Killing field equations.
The  log of tau function 
is defined as the central component of the affine extension of the  spectral Killing field.
We give a closed formula for tau function in terms of the determinant of the spectral Killing field.
\end{abstract}
\maketitle

%-----------------------------------------------------------------------------%
\setcounter{tocdepth}{2}
\tableofcontents

%-----------------------------------------------------------------------------%
%\listoftables   %if you have any tables

%-----------------------------------------------------------------------------%
%\listoffigures  %if you have any figures

%-----------------------------------------------------------------------------%
%
% MAIN BODY OF PAPER
%
% replace TEXT in \chapter{TEXT} with actual chapter name
% replace FILE in \input{FILE} with name of text file
%   containing the given chapter, eg. for the introduction one could
%   have FILE = intro IF stored in intro.tex (.tex extension is assumed!).
%
%%%%%%%%%%%%%%%%%%%%%%%%%%%%%%%%%%%%

%%%%%%%%%%
%%%%%%%%%% 
%%%%%%%%%% 
%%%%%%%%%% 
%%%%%%%%%%%%%%%%%%%%
\section{Introduction}\label{sec:intro}
This is the second part of the series on CMC hierarchy.
%---------------------------------------------------------
\sub{Classical spectral symmetry}
In the previous work \cite{Wang2013},
an analysis of the infinite jet space of the differential system for CMC surfaces
revealed the canonically defined notion of  spectral weight.
It was observed through various applications 
that the corresponding weighted homogeneity 
is an important and useful aspect of the structural property of the CMC system.

In order to give this notion a proper treatment, 
the associated  classical  \tb{spectral symmetries} (vector fields) were introduced,
for which the spectral weights are the weights under Lie derivatives.
The classical spectral symmetries are the non-local objects
defined on an integrable extension, and by construction
they form an affine space over the classical symmetries.

%---------------------------------------------------------
\sub{Spectral Killing field}
In this paper, we show that there exist  the  higher-order spectral symmetries
for a truncated\ftmark\fttext{See \S\ref{sec:introtruncated} below.} CMC hierarchy  %%%%%  
which are generated by \tb{spectral Killing fields}. 
Similarly as  the classical spectral symmetries,
the spectral Killing fields form an affine space over the formal Killing fields.
In fact, in terms of the Lie bracket relations with the formal Killing fields,
we will be able to make a preferred choice of the normalized spectral Killing field.

The normalized spectral Killing field is another important invariant  of a truncated CMC hierarchy.
One of the main results of this paper is to give a concrete realization of 
the higher-order spectral symmetries generated by the  normalized spectral Killing field
as the pair of  truncated Virasoro algebras of (non-commuting) 
symmetry vector fields on the formal moduli space of solutions to
a truncated CMC hierarchy. 

This is achieved by extending a truncated CMC hierarchy
to a system of equations for the associated affine Kac-Moody algebra valued  Killing fields.
By a sequence of natural decompositions,  
similarly as in the AKS construction of the original CMC hierarchy,
the normalized spectral Killing field provides  the relevant deformation coefficients
and in this sense generates the  extension.
%---------------------------------------------------------
\subb{Truncated CMC hierarchy}\label{sec:introtruncated}
%A technical remark is in order.
To avoid some of the formal series vs. convergence related issues, 
and to stay within the realm of differential algebraic analysis,
we shall truncate the time variables
and work with a truncated CMC hierarchy, \S\ref{sec:truncated}.
The truncation integer parameter ``$N$" also determines 
the corresponding truncation of the associated  (centerless) Virasoro algebras, \S\ref{sec:hVir}.
See \S\ref{sec:formalvs} for the related remarks.
 
%---------------------------------------------------------
\sub{Related works}
The Virasoro algebra type of non-commuting symmetries 
for the integrable equations of KP type
are discussed in \cite{Dickey2003}  in the context of additional symmetries.
For the more recent works, we refer to \cite{Wu2012,Wu2013} and the references therein.

For the physics literature on the extended symmetries of  coset models,
we refer to \cite{Nicolai1991b} and the references therein.
For the string theory aspects of the extended symmetries of  integrable equations, 
we refer to the survey \cite{Moerbeke1994}. 
Witten gave the original formulation of the conjecture that
the partition function for a certain two dimensional model of quantum gravity
is a tau-function for KdV hierarchy, 
which is subject to the additional  Virasoro constraints, \cite{Witten1991,Kontsevich1992}.
The motivation for the ansatz for the affine extension of a truncated CMC hierarchy
is drawn from \cite{Semikhatov1991}.

Recently, Terng and Uhlenbeck gave 
a comprehensive analysis of the Virasoro symmetries
for a class of matrix Lax equations,
\cite{Terng2014b,Terng2014a}.
The main difference is that 
they apply the loop group action based on loop group factorization,
whereas we apply the dressing transformation based on  loop algebra decomposition.

\two
In  \cite{Wu2012} cited above, Wu gave a  uniform construction of the tau functions 
for  Drinfeld-Sokolov hierarchies via  the certain generating functions of Hamiltonian densities.
The underlying idea of this construction agrees with our definition, Defn.\ref{defn:tau}.
%It is expected that these two definitions are equivalent.
For the related works, we refer to \cite{Miramontes1999}, \cite{Terng2014b,Terng2014a} 
and the references therein.

%---------------------------------------------------------
\sub{Results}
%---------------------------------------------------------
\subb{Spectral Killing field}
Via a canonical dressing transformation, we define the wave function for 
a  truncated  CMC hierarchy, Defn.\ref{defn:wave}.
Based on this idea,  we determine the space of (formal) solutions to 
the associated Killing field equation, 
Eq.\eqref{eq:PPpm}, Thm.\ref{thm:PP}. 
The   Killing fields satisfy the $\sla(2,\C)$-type of  Lie bracket relations, 
and
this observation leads to an algebraic formula (up to an integrable extension)  
for the normalized spectral Killing field, Prop.\ref{prop:spectral}.
%---------------------------------------------------------
\subb{Extended CMC hierarchy}
The relevant Lie algebra for our analysis is the generalized affine Kac-Moody algebra $\hbmg^{\pm}_{N+1}$,
which is the twisted loop algebra $\bmg$, \eqref{eq:twisted},
enhanced by the truncated Virasoro algebra of even derivations $\Vir_{N+1}^{\pm}$,
Defn.\ref{defn:truncatedVir}.
We propose an extension of a truncated CMC hierarchy
%by the non-commuting Virasoro symmetries generated by the normalized spectral Killing field
in terms of the $\hbmg^{\pm}_{N+1}$-valued  affine Killing field equations,
Defn.\ref{defn:eCMC}.

The compatibility of the extended CMC hierarchy is proved 
by a direct computation, Thm.\ref{thm:main}. 
This in particular shows that the Virasoro symmetries commute with the higher-order symmetries
of the CMC system.
The compatibility also reflects  that the differential algebraic relations among 
the formal Killing fields, the normalized spectral Killing field, and
the Maurer-Cartan form for a truncated CMC hierarchy 
are consistent with the Lie algebra structure of  $\hbmg^{\pm}_{N+1}$. 
%---------------------------------------------------------
\subb{Tau function}
The log of tau function (denoted by $\tau$) for the extended CMC hierarchy is defined
as the central component of the affine extension of the spectral Killing field, Defn.\ref{defn:tau}.
We give a closed formula for $\log(\tau)$, Thm.\ref{thm:tau}. It states that
\begin{align}
\log(\tau) =\tn{Res}_{\lambda=0}&\big[ \big(\tn{Killing field for Virasoro algebra}\big) \\
&\;\times \big(\det(\tn{normalized spectral Killing field})\big) \big].\n
\end{align}
Here   $\tn{Res}_{\lambda=0}$
is the residue operator that takes the terms of $\lambda$-degree 0.
%---------------------------------------------------------
%\subb{Tau function}
%\marg{variation of cv laws}

%---------------------------------------------------------
\sub{Contents}
In \S\ref{sec:formulas}, 
we give  a  summary of the relevant  formulas from Part \Rmnum{1}.
In \S\ref{sec:dressing}, 
we introduce a canonical dressing transformation.
In \S\ref{sec:addKilling}, 
we determine  a pair of additional  formal Killing fields.
Based on this, in \S\ref{sec:spectral} 
we give an algebraic formula for the normalized spectral Killing field.
In \S\ref{sec:affineKilling},
we show that the original formal Killing field   and the normalized spectral Killing field  
admit the $\hbmg^{\pm}_{N+1}$-valued affine lifts.
In \S\ref{sec:extension},
we examine an ansatz for the  affine extension of the structure equation 
for the formal Killing field,
and show that it is compatible.
In \S\S\ref{sec:CMC+},\ref{sec:addaffine},\ref{sec:exformulas},
we give a definition of the extended CMC hierarchy,
and subsequently in \S\ref{sec:proof} 
show that the extended structure equation is compatible.
In \S\ref{sec:central},
we  give a closed formula for tau function  
in terms of the determinant of the normalized spectral Killing field.
  
%-----------------------------------------------------------------------
%-----------------------------------------------------------------------
\section{Formulas from Part \Rmnum{1}}\label{sec:formulas}
We recall the relevant formulas from Part \Rmnum{1} of the series. 
To avoid repetition, we only state the title of the formulas
and refer the reader to Part \Rmnum{1}.

In the structure equations recorded below, the equality sign ``$ = $'' would mean
``$ \equiv $" modulo the appropriate differential ideal.
The meaning will be clear from the context, and 
we shall omit the specific description of the details.

\two\one
\noi
[\tb{$\sla(2,\C)[[\lambda]]$-valued formal Killing field}]
\be\label{eq:tbY}%\tag{eq:tbY}
\tb{Y}=\bp -\im\tb{a}&2\tb{c}\\ 2\tb{b}&\im\tb{a} \ep,  
\ee
\be
\tb{a}=\sum_{n=0}^{\infty}\lambda^{2n}a^{2n+1},\qquad
\tb{b}=\sum_{n=0}^{\infty}\lambda^{2n+1}b^{2n+2},\qquad
\tb{c}=\sum_{n=0}^{\infty}\lambda^{2n+1}c^{2n+2}.\n
\ee

\be\label{eq:b2c2}
b^2=-\im\gamma h_2^{-\frac{1}{2}}, \quad c^2=\im h_2^{\frac{1}{2}}, \quad
\det(\bY)=-4b^2 c^2\lambda^2=-4\gamma\lambda^2.
\ee

\noi
[\tb{Recursive structure equation for the coefficients of  $\bY$ \tn{(for the CMC system)}}]
\begin{align}\label{eq:abcstrt}%\tag{eq:abcstrt}
\ed  a^{2n+1}&=  (\im \gamma c^{2n+2}+\im h_2 b^{2n+2})\xi + (\im \gamma  b^{2n}+\im \hb_{2} c^{2n})\xib, \\
\ed  b^{2n+2} - \im  b^{2n+2} \rho&= \frac{\im \gamma }{2} a^{2n+3} \xi + \frac{\im}{2}  \hb_{2}a^{2n+1} \xib ,\n\\
\ed  c^{2n+2} +\im c^{2n+2}  \rho&= \frac{\im}{2}  h_2 a^{2n+3}\xi  + \frac{ \im \gamma }{2} a^{2n+1} \xib.\n
\end{align}
  
\noi
[\tb{Decomposition of $\bY$}]
\be\label{eq:YUm}%\tag{eq:YUm}
\bY=2\im\lambda^{2m+2} (U_m+U_{(m+1)}).
\ee
\be
U_m =\bp -\im U_m^a &2U_m^c \\ 2U_m^b&\im U_m^a\ep,\n
\ee
\begin{align*}%\label{eq:Uabc}%\tag{eq:Uabc}
U_m^a&=\frac{1}{2\im}\sum_{j=0}^m\lambda^{(2j+0)-(2m+2)}a^{2j+1}, \quad
U_m^c =\frac{1}{2\im}\sum_{j=0}^m\lambda^{(2j+1)-(2m+2)}c^{2j+2},\\
U_m^b&=\frac{1}{2\im}\sum_{j=0}^m\lambda^{(2j+1)-(2m+2)}b^{2j+2}. \n
\end{align*}

\iffalse
\be\label{eq:bYm}%\tag{eq:bYm}
U_{(m+1)}=\bp -\im U^a_{(m+1)}&2U^c_{(m+1)}\\ 2U^b_{(m+1)}&\im U^a_{(m+1)} \ep 
\ee

\begin{align}\label{eq:Upabc}%\tag{eq:Upabc}
U_{m+1}^a&=\frac{1}{2\im}\sum_{j=m+1}^{\infty}\lambda^{(2j+0)-(2m+2)}a^{2j+1}, \quad
U_{m+1}^c =\frac{1}{2\im}\sum_{j=m+1}^{\infty}\lambda^{(2j+1)-(2m+2)}c^{2j+2},\\
U_{m+1}^b&=\frac{1}{2\im}\sum_{j=m+1}^{\infty}\lambda^{(2j+1)-(2m+2)}b^{2j+2}. \n
\end{align} 
\fi

\iffalse
\be
\tb{a}_{(m+1)}=\sum_{n=m+1}^{\infty}\lambda^{2n}a^{2n+1},\qquad
\tb{b}_{(m+1)}=\sum_{n=m+1}^{\infty}\lambda^{2n+1}b^{2n+2},\qquad
\tb{c}_{(m+1)}=\sum_{n=m+1}^{\infty}\lambda^{2n+1}c^{2n+2}.\n
\ee
\fi
%------------------------------------------------------------------------

\noi
[\tb{$\sla(2,\C)[[\lambda^{-1},\lambda]]$-valued Maurer-Cartan form}]
\begin{align}\label{eq:phi2} %\tag{eq:phi2} 
\phi_+&=\bp \cdot& -\frac{1}{2}\gamma \xib\\ \frac{1}{2} \hb_2 \xib &\cdot \ep,\qquad
\phi_0 =\bp \frac{\im}{2}\rho &\cdot \\ \cdot&-\frac{\im}{2}\rho\ep, \qquad
\phi_- =\bp \cdot& -\frac{1}{2}  h_2 \xi\\  \frac{1}{2}\gamma \xi&\cdot \ep, \\
\phi_{\lambda}&=\lambda\phi_+ + \phi_0 + \lambda^{-1}\phi_-
=\bp  \frac{\im}{2}\rho &-\lambda\frac{1}{2}\gamma \xib-\lambda^{-1}\frac{1}{2}  h_2 \xi \n \\
\lambda \frac{1}{2} \hb_2 \xib+ \lambda^{-1}\frac{1}{2}\gamma \xi  &-\frac{\im}{2}\rho\ep.\n
\end{align}
\begin{align}\label{eq:extphi2}%\tag{eq:extphi2}
\bphi =\bphi_++\bphi_0+\bphi_-  
:&=-\sum_{m=0}^{\infty}\ol{U}^t_m\ed t_{\ol{m}}+\phi_{0}+\sum_{m=0}^{\infty}U_m\ed t_m   \\
&=-\sum_{m=1}^{\infty}\ol{U}^t_m\ed t_{\ol{m}}+\phi_{\lambda}+\sum_{m=1}^{\infty}U_m\ed t_m.\n
\end{align}

\noi
[\tb{Structure equation for CMC hierarchy}]
\be
%\tag{Eq.($\bphi$)}
\label{Eqbphi}\left\{
\begin{array} {rl}
\ed\xi-\im\rho\w\xi     &=\sum_{m=1}^{\infty} a^{2m+3}\ed t_m\w\xi, \\
\ed\xib+\im\rho\w\xib &=\sum_{m=1}^{\infty} \ab^{2m+3}\ed \tba_{m}\w\xib,   \\
\ed\rho &\equiv R\frac{\im}{2}\xi\w\xib  \qquad \mod \ed\bbt, \ed\ol{\bbt},\\
%\sum_{m=1}^{\infty}\ed t_m\w\left((\gamma b^{2m+2}+\hb_2 c^{2m+2})\xib\right),  \\
\ed h_2+2\im h_2\rho&=h_3\xi -2 \sum_{m=1}^{\infty}h_2 a^{2m+3}\ed t_m,\\
\ed \hb_2-2\im \hb_2\rho &=\hb_3\xib-2 \sum_{m=1}^{\infty}\hb_2 \ab^{2m+3}\ed \tba_{m}.
\end{array}\right.\ee
\iffalse
\[
\underbrace{\left(\ed \phi_{\lambda}+ \phi_{\lambda}\w \phi_{\lambda}\right)}_\text{B$^1$}
+\underbrace{\sum_{m=1}^{\infty}\left(\ed U_m +[ \phi_{\lambda},U_m]\right)\w \ed t_m}_\text{B$^0$}
+\underbrace{\sum_{m,\ell=1}^{\infty}\frac{1}{2}[U_m,U_{\ell}] \ed t_m\w\ed t_{\ell}  \tag{eq:B}}_\text{B$^{-2}$}=0.
\]
\fi
\begin{align}\label{eq:Ybphi}
\ed\tb{Y}+[\bphi,\tb{Y}]&=0,  \\
\ed\bphi+\bphi\w\bphi&=0. \n
\end{align}

\noi
[\tb{Twisted loop algebra}]
\begin{align}
\g&=\sla(2,\C),\n\\
\g((\lambda)):&=\{\tn{$\g$-valued formal Laurent series in $\lambda$}\},\n\\
\label{eq:twisted}%\tag{eq:twisted}
\bm{\g}:&=\left\{   \, h(\lambda)\in \g((\lambda))\;\vert\;\,\tn{$h^1_1(\lambda)=-h^2_2(\lambda)$ is even in $\lambda$; $h^1_2(\lambda), h^2_1(\lambda)$ are odd in $\lambda$}     \,   \right\}.
\end{align}

\noi
[\tb{Decomposition of $\bmg$}]
\be\begin{array}{rll}\label{eq:vsdecomp}%\tag{eq:vsdecomp}
\bm{\g}&=\;\bm{\g}_{\leq -1} &+^{vs}\quad \bm{\g}_{\geq 0} \\
&\subset\;\g[\lambda^{-1}]\lambda^{-1}&+^{vs}\quad\g[[\lambda]]. 
\end{array}\ee

\be\begin{array}{rll}\label{eq:vsdecompdual}%\tag{eq:vsdecompdual}
\bm{\g}^*=\bm{\g}
& =\;\bm{\g}_{\geq  1} &+^{vs}\quad \bm{\g}_{\leq 0}\\
&\subset\;\g[[\lambda]]\lambda &+^{vs}\quad\g[\lambda^{-1}].  \\
\end{array}\ee

\noi
[\tb{Lie groups}]
\begin{center}
 $\bm{G}, \bm{G}_{\leq -1}, \bm{G}_{\geq 1}, 
\bm{G}_{\leq 0}, \bm{G}_{\geq 0}$: (formal) Lie groups for  $\bmg, \bmg_{\leq -1}, \bmg_{\geq 1},
\bmg_{\leq 0}, \bmg_{\geq 0}.$
\end{center}

\iffalse
Let $\Lambda G=\Lambda\SL_2(\C)$ be the loop group of $\SL_2(\C)$.  Let $\Lambda\g=\Lambda\sla_2(\C)$ be its Lie algebra.
Let $\lambda\in\C^*\subset\C\cup\{\infty\}=\C\PP^1$ 
be the corresponding spectral parameter.
Define the twisted loop group by
\begin{align}
\Lambda G'=\{\, g \in\Lambda G\,\vert\,
&\tn{\quad \;  diagonal components of $g$ are even functions of $\lambda$}, \n \\
&\tn{off-diagonal components of $g$ are odd functions of $\lambda$}\,\}. \n
\end{align}
Let $\Lambda\g'\subset\Lambda \g$ denote its Lie algebra.
Let 
\begin{align}
D^+&=\{ \, \lambda\in\C\PP^1\; \vert \quad\vert\lambda\vert\leq 1\,\},\n\\
D^-&=\{ \, \lambda\in\C\PP^1\; \vert \quad \vert\lambda\vert\geq 1\,\}.\n
\end{align}
Define the subgroups
\begin{align}
\Lambda^+G' &=\{\, g \in\Lambda G'\,\vert\,
g(\lambda) \,\tn{admits a holomorphic extension to $D^+$ }\,\},\n\\
\Lambda_I^-G' &=\{\, g \in\Lambda G'\,\vert\,
g(\lambda)  \,\tn{admits a holomorphic extension to $D^-$},\; g(\infty)=I_2 \,\}.\n 
\end{align}
Here $I_2$ is the 2-by-2 identity matrix.
Let $\Lambda^+\g', \Lambda_0^-\g' $ be the corresponding subalgebras respectively,
\begin{align}
\Lambda^+\g'&=\{\, A \in\Lambda\g'\,\vert\,
A(\lambda) \,\tn{admits a holomorphic extension to $D^+$ }\,\},\n\\
\Lambda_0^-\g' &=\{\, A \in\Lambda\g'\,\vert\,
A(\lambda)  \,\tn{admits a holomorphic extension to $D^-$},\; A(\infty)=0 \,\}.\n 
\end{align}
\fi

%----------------------------------------------------------------------------
\sub{Truncated CMC hierarchy}\label{sec:truncated}
The construction of the extended CMC hierarchy will require that 
the base CMC hierarchy is $ \ol{\bbt}, \bbt$-truncated as follows;
once and for all, set the non-negative integer $N\geq 0$ and assume
$$
\fbox{$\qquad\tba_m, t_m =0,\quad \forall \,m\geq N+1.\qquad$}
$$ 
Under this condition, the CMC hierarchy is called \emph{ $\tba_N, t_N$-truncated}.
From now on, we will only consider the truncated CMC hierarchy.

Note in this case that 
the $\lambda$-degree  of $\bphi$ is bounded in the interval 
$$[-(2N+1), 2N+1].$$
 
%%%%%%%%%%%%%%%
%%%%%%%%%%%%%%%
%%%%%%%%%%%%%%%
%%%%%%%%%%%%%%%
%%%%%%%%%%%%%%%%%%%%%%%%%%%%%%
\section{Dressing}\label{sec:dressing}
In this section,  we start by introducing a canonical dressing transformation
which renders the CMC hierarchy 
into a completely integrable (Frobenius)  system of 
constant coefficient linear partial differential equations, 
see  \cite{Mulase1984} for the related  work on KP hierarchy.
This leads to the wave function for the CMC hierarchy. 

The wave function formulation will play an important role 
in finding the additional Killing fields, \S\ref{sec:addKilling}.
%----------------------------------------------------------------------------
%----------------------------------------------------------------------------
\sub{Dressing transformation}\label{sec:dressingtransform}
Consider the decomposition \eqref{eq:YUm} of the formal Killing field $\bY$,
and the formula \eqref{eq:extphi2} for the Maurer-Cartan form $\bphi$.
%----------------------------------------------------------------------------
\subb{Maurer-Cartan form for dressing}
Set
\be\label{eq:cbphi} 
\cbphi:=\bphi_++\bphi_0-\sum_{m=0}^{\infty}U_{(m+1)}\ed t_m.
\ee
Note that $\cbphi$ is $\bmg_{\geq 0}$-valued. 
By \eqref{eq:YUm},  we  have the identity
\be\label{eq:Yidentity} 
\bphi-\cbphi=\bY\balpha,
\ee
where
$$\balpha:= \frac{1}{2\im}\sum_{m=0}^{\infty}\lambda^{-(2m+2)}\ed t_m.
$$
Note that $\ed\balpha=0.$
\begin{lem}\label{eq:cbphicompat}
The $\bmg_{\geq 0}$-valued 1-form $\cbphi$ satisfies the Maurer-Cartan equation
$$\ed\cbphi+\cbphi\w\cbphi=0.
$$
\end{lem}
\begin{proof}
From the equation $\cbphi=\bphi-\bY\balpha$, 
$$\ed\cbphi+\cbphi\w\cbphi=\left(-\bphi\w\bphi+ (\bphi\bY-\bY\bphi)\w\balpha\right)
+\left(\bphi\w\bphi-\bphi\w\bY \balpha-\bY\balpha\w\bphi\right)=0.
$$
\end{proof}
%----------------------------------------------------------------------------
\subb{Dressing transformation}
Let $S$ be a $\bm{G}_{\geq 0}$-valued frame for $\cbphi$, 
which satisfies the equation\ftmark\fttext{Here 
we need the assumption that the CMC hierarchy is $\ol{\bbt}, \bbt$-truncated.}%-------------
$$S^{-1}\ed S=\cbphi.
$$
The frame $S$ is determined up to left multiplication by $\bm{G}_{\geq 0}$.

Consider the dressing transformation of the  Maurer-Cartan form $\bphi$
by the  $\bm{G}_{\geq 0}$-frame $S^{-1}$.
We have
\begin{align}
S\bphi S^{-1}-(\ed  S) S^{-1}
&=S\left(\bphi -S^{-1}\ed S\right)S^{-1}\n\\
&=S\left(\bphi-\cbphi\right)S^{-1}\n\\
&=S \bY S^{-1}\balpha=\bZ\balpha, %\quad\tn{where}\; \bZ:=S\bY S^{-1}.\n
\end{align}
where $\bZ:=S\bY S^{-1}.$ 
\begin{lem}
Under the dressing by $S^{-1}$,
\begin{enumerate}[\qquad a)]
\item
the structure equation for the CMC hierarchy transforms as
$$\ed W=W\bphi\quad \lra\quad \ed w=w \bZ\balpha.$$
\item
the Killing field equation transforms as
$$\quad\ed \bP+[\bphi,\bP]=0\quad \lra \quad\ed\bQ+[\bZ\balpha,\bQ]=0.$$
\item
the formal Killing field $\bY$ transforms to
a constant element $\bZ=S\bY S^{-1} \in \bmg_{\geq 1}$.
\end{enumerate}
\end{lem}
\begin{proof}
Since $\bY$ satisfies the Killing field equation $\ed\bY+[\bphi,\bY]=0$,
$\bZ$ satisfies the dressed equation
$$\ed\bZ+[\bZ\balpha, \bZ]=\ed\bZ=0.
$$
Note $\bY\vert_{\lambda=0}=0,$ and hence $\bZ\vert_{\lambda=0}=0.$
\end{proof}
Note that, since $S$ is determined up to left multiplication by $\bm{G}_{\geq 0}$,
$\bZ$ is determined up to conjugation by $\bm{G}_{\geq 0}.$
\begin{rem}\label{rem:detZ}
\be\label{eq:detZ}
\det(\bZ)=\det(\bY)=-4\gamma\lambda^2.
\ee
\end{rem}
%----------------------------------------------------------------------------
\sub{Wave function}\label{sec:wavefunction}
Let 
\be\label{eq:ttt}
\ttt:=\int\balpha=\frac{1}{2\im}\sum_{m=0}^{\infty} \lambda^{-(2m+2)}t_m.
\ee
Set 
\be\label{eq:bzeta}
\bzeta:= \int \bZ\balpha=\bZ\ttt 
\ee be an anti-derivative for $\bZ\balpha$.
Note $\del_{\tba_n} \bzeta=0, \;\forall \,n\geq 0.$
\begin{defn}\label{defn:wave}
The \tb{wave function} (or Baker-Akhiezer function) for the CMC hierarchy is
$$W:= e^{\bzeta} S.
$$
\end{defn}
\begin{thm}\label{thm:dressing}
The wave function $W$ satisfies the structure equation for the CMC hierarchy,
\be\label{eq:Wstrt}
\ed W=W\bphi.
\ee
\end{thm}
\begin{proof}
The exponential factor $e^{\bzeta}$ is a solution to the dressed structure equation (by $S^{-1}$),
$$ \ed e^{\bzeta} = e^{\bzeta} \bZ\balpha.$$
\end{proof}
%----------------------------------------------------------------------------
%----------------------------------------------------------------------------
\sub{Normalization of $\bZ$}\label{sec:normalization}
We record  the   asymptotics of the dressing matrix $S$ and the dressed formal Killing field $\bZ$
at $\lambda=0$. This will lead to a  normalization of $\bZ$.
%----------------------------------------------------------------------------
\subb{Asymptotics}
From  \eqref{eq:cbphi} and the structure equation \eqref{Eqbphi},
one finds that
$$\cbphi\vert_{\lambda=0}=\bp \ed\log(h_2^{-\frac{1}{4}})&\cdot\\ \cdot &\ed\log(h_2^{+\frac{1}{4}})\ep.$$
Consider the expansion
$$S=\Big( I_2+S_1\lambda+\mco(\lambda^2)\Big)\bp  h_2^{-\frac{1}{4}} 
&\cdot\\ \cdot & h_2^{+\frac{1}{4}} \ep,$$
where $I_2$ denotes the 2-by-2 identity matrix.
Substitute this to $\bZ=S\bY S^{-1}$, and one finds
$$\bZ= Z\lambda+\mco(\lambda^2),$$
where
$$Z:=\bp \cdot & 2\im \\-2\im\gamma&\cdot\ep$$
is the leading coefficient matrix.
The constants $2\im, -2\im\gamma$ in $Z$ follow from  \eqref{eq:b2c2}.
%----------------------------------------------------------------------------
\subb{Normalization of $\bZ$}
Recall that  $\bZ$ is determined up to conjugation by $\bm{G}_{\geq 0}$.
Since the leading coefficient matrix $Z$ is non-degenerate,  
one may normalize $\bZ$ by  the formal adjoint action by $\bm{G}_{\geq 1}$ 
(to keep the leading term $Z$ unchanged)
such that 
\be\label{eq:Znormal}\bZ= Z\lambda.\ee
We assume this normalization of $\bZ$ from now on.
%Note for example that
%\be\label{eq:bzetanew}
%\bzeta=\bZ\ttt=Z\frac{1}{2\im} \sum_{m=0}^{\infty}\lambda^{-(2m+1)}t_m.
%\ee

\iffalse%%%%%%%%%%%%%%%%%%%%%%%%
 
\begin{prop}[Bilinear identity]\label{prop:bilinear}
The identity
$$\tn{Res}_{\lambda^{-1}}
\left[ W^{-1} (\del^{i_0}_{t_0}\del^{i_1}_{t_1}\, ... \, \del^{i_m}_{t_m} W)  \right]=0.
$$
holds for any $m\geq 0$, and  $(i_0, i_1, \, ... \, i_m)$, $i_{k}\geq 0$.
Here $\tn{Res}_{\lambda^{-1}}(\cdot)$ indicates the $\lambda^{+1}$ coefficient.
%As a corollary, we have
%\be\label{eq:bilinear}
 %\tn{Res}_{\lambda^{-1}}
%\left[ W^{-1}(t'_0, \tb{t}';\lambda) W(t_0, \tb{t};\lambda)\right]=0.
%\ee
%Here $\tb{t}=(t_1, t_2, \, ... \, ).$
\end{prop}
\begin{proof}
By Eq.\eqref{eq:Wstrt}, we have
$$\del_{t_m}W=W U_m.$$
The claim follows for the $\lambda$-degree of $U_m$ terms are bounded above by $-1.$
\end{proof}
%Note that the function $\ol{t}_0$ is not in general the complex conjugate of $t_0$.
\fi

\iffalse
Some of the formal computations in this section would make sense when one restricts to a finite set of time variables $t_m$.
\begin{defn}
For an integer $N\geq 0$,
the \tb{$t_N$-truncated} CMC hierarchy is the CMC hierarchy under the constraint that 
$$t_m=0\;\forall m\geq N+1.$$
\end{defn}
\fi 

%\np
%%%%%%%%%%%%
%%%%%%%%%%%%
%%%%%%%%%%%%
%%%%%%%%%%%%
%%%%%%%%%%%%%%%%%%%%%%%%
\section{Additional Killing fields}\label{sec:addKilling} 
We determine a pair of  (formal) solutions to the Killing field equation 
for each $\cbphi$ and $\bphi$ via an elementary ansatz based on
the eigen-decomposition for the adjoint operator $\tn{ad}_{\bY}$.  

The formal Killing field $\bY$ and these additional Killing fields satisfy 
the $\sla(2,\C)$-type of  Lie bracket relations, \eqref{eq:YBRVV}, \eqref{eq:YPP}.
This will lead to an essentially algebraic formula 
for the normalized spectral Killing field, \S\ref{sec:spectral}.
  
%----------------------------------------------------------------------------
%----------------------------------------------------------------------------
\sub{Killing fields for $\cbphi$}
We first determine the  Killing fields for $\cbphi.$
The  Killing fields  for $\bphi$ will be obtained from these by a certain linear transformation.

Observe that
\be\label{eq:Ycbphi}
\ed\bY+[\cbphi,\bY]=\ed\bY+[\bphi-\bY\alpha,\bY]=0,\ee
and $\bY$ is also a Killing field for $\cbphi.$
%----------------------------------------------------------------------------
\subb{Ansatz}
Considering the eigen-matrices of the operator $\ad_{\bY}$ on $\bmg$, 
set  
\begin{align}\label{eq:VVpm}
\bV_+&=\bp  2u\bbb\bbc&\im u \bba\bbc-2\sqrt{\gamma}v\bbc \\ 
\im u \bba\bbb+2\sqrt{\gamma}v\bbb&-2u\bbb\bbc\ep, \\
\bV_-&=\bp  2v\bbb\bbc& \im v \bba\bbc-2\sqrt{\gamma}u\bbc\lambda^2 \\ 
 \im v \bba\bbb+2\sqrt{\gamma}u\bbb\lambda^2& -2v\bbb\bbc\ep.\n
\end{align}
The coefficients $u, v$, which are to be determined, 
are $\C[[\lambda^2]]$-valued (even) functions.

The set of $\bmg_{\geq 0}$-valued functions $\{\bY, \bV_{\pm}\}$ 
satisfy the following Lie-bracket relations,
which resemble the structure equation of $\sla(2,\C)$:
\begin{align}\label{eq:YBRVV}
[\bY,\bV_+]&= 4\sqrt{\gamma}  \bV_-, \qquad
[\bY,\bV_-]= 4\sqrt{\gamma}\lambda^2 \bV_+, \\
&\quad[\bV_+,\bV_-]=\mu\sqrt{\gamma} \bY,\n
\end{align}
where 
$$\mu=4\bbb\bbc(v^2-u^2\lambda^2).$$ 
It will be shown that $\mu$ is necessarily a constant element in  $\C[[\lambda^2]]$.

Note also that
$$\det(\bV_+)=4\gamma\bbb\bbc(v^2-u^2\lambda^2),\qquad
\det(\bV_-)=-4\gamma\bbb\bbc\lambda^2(v^2-u^2\lambda^2).$$
%Observe that under the substitution $(u,v) \ra (v,u\lambda^2)$, we have
%$$(\bV_+,\bV_-)\quad\lra\quad (\bV_-,\lambda^2\bV_+).$$
\begin{rem} 
Note $\tr(\bY\bV_{\pm}), \tr(\bV_+\bV_-)=0.$
\end{rem}
%-------------------------------------------------------------- 
\subb{Associated  linear differential system}
Substitute \eqref{eq:VVpm} to the  Killing field equation for $\cbphi$,
$$\ed\bV_{\pm}+[\cbphi,\bV_{\pm}]=0.$$
After collecting terms,
this reduces to the following  linear system of differential equations for $(u, v)$:
\be\label{eq:EQuv}
\ed(u, v)=(u,v)\Omega,
\ee
where 
\be\label{eq:Omega}
\Omega=\frac{1}{\bbb\bbc}\begin{bmatrix}
\frac{\im}{2}\bba\left(\bbb\cbphi^1_2-\bbc\cbphi^2_1\right) &
-\sqrt{\gamma}\lambda^2 \left(\bbb\cbphi^1_2+\bbc\cbphi^2_1\right) \\
-\sqrt{\gamma} \left(\bbb\cbphi^1_2+\bbc\cbphi^2_1\right)&
 \frac{\im}{2}\bba\left(\bbb\cbphi^1_2-\bbc\cbphi^2_1\right)\end{bmatrix}.
\ee
Here $\cbphi^i_j$ denotes the $(i,j)$-component of  $\cbphi$.

It is easily checked from this that  $\ed \mu=0$.
%and $\det(\bV_{\pm})$ are necessarily constant.
   
\two
We proceed to solve \eqref{eq:EQuv}.
The matrix valued 1-form $\Omega$ can be decomposed as
$$\Omega=\bp 1&\cdot\\ \cdot&1\ep \btheta^+
+\bp \cdot &1 \\ \lambda^{-2} &\cdot \ep\btheta^-,$$
where
\begin{align}\label{eq:bthetapm}
\btheta^+&:=\frac{\im}{2}\bba\left(\frac{\bbb\cbphi^1_2-\bbc\cbphi^2_1}{\bbb\bbc}\right),  \\
\btheta^-&:=-\sqrt{\gamma}\lambda^2\left(\frac{\bbb\cbphi^1_2+\bbc\cbphi^2_1}{\bbb\bbc}\right).\n
\end{align}
Note that $\btheta^{\pm}$ are both even as a function of $\lambda$.

Eq.\eqref{eq:Ycbphi} and the compatibility of $\cbphi$  show that (we omit the details of computation)
$$\ed\btheta^{\pm}=0.$$  
Since $\Omega\w\Omega=0$ and $[\Omega, \int\Omega]=0$, consider
$$g=\exp(\int\Omega).$$
Then the matrix valued function $g$ solves the equation $$\ed g=g\Omega.$$
Up to left multiplication of $g$ by $\bm{G}_{\geq 0}$,  
the desired 2-dimensional space of solutions $(u,v)$ for \eqref{eq:EQuv} 
are generated  by the rows of $g$.
 
Set
$$\bsigma^{\pm}:=\int\btheta^{\pm}.$$
Then it is easily checked that
\be\label{eq:gformula}
g=\exp(\int\Omega)=e^{\bsigma^+}
\bp\cosh(\lambda^{-1}\bsigma^-)&\lambda\sinh(\lambda^{-1}\bsigma^-) \\ 
\lambda^{-1}\sinh(\lambda^{-1}\bsigma^-)& \cosh(\lambda^{-1}\bsigma^-)\ep.
\ee
 
%-------------------------------------------------------------- 
\subb{Local function $\bsigma^+$}
\begin{lem}\label{lem:sigmap}
$$\btheta^+=-\frac{1}{2}\ed\log(\bbb\bbc).$$
\end{lem}
\begin{proof} 
It follows from \eqref{eq:Ycbphi}.
\end{proof}
Hence, up to translating $\bsigma^+$ by a constant element in $\C[[\lambda^2]]$, 
we may set
$$e^{\bsigma^+}=\frac{1}{2\sqrt{\bbb\bbc}}$$
so that
\be\label{eq:bsigma+normal}
 4 e^{2\bsigma^+}\bbb\bbc=1.\ee
We assume this normalization of $\bsigma^+$ from now on.
  
%-------------------------------------------------------------- 
\subb{Non-local function $\bbp$}
Let us put $\bsigma^-=-\lambda^2\bbp$, where
\be\label{eq:bbp}
\ed\bbp:=
\sqrt{\gamma}\left(\frac{\bbb\cbphi^1_2+\bbc\cbphi^2_1}{\bbb\bbc}\right).
\ee
Expanding  as a series in $\lambda$, one finds
$$\ed\bbp=  \left( \frac { \left(  h_{{2}} b_{{4}}
-\gamma  c_{{4}} \right) \xi }{2\sqrt {\gamma}  }
+ \frac { \im \left(  h_{{2}}\hb_2+ {\gamma}^{2} \right) \xib
}{2\sqrt {\gamma}h_2^{\frac{1}{2}} }\right)+\mco(\lambda^2)\mod \ed \tb{t},\ed\ol{\tb{t}}.
$$
This suggests that $\ed\bbp$ contains all the higher-order conservation laws,  
and hence that $\bbp$ is indeed a non-local $\C[[\lambda^2]]$-valued function.
 
%-------------------------------------------------------------- 
\subb{Formulas for $\bV_{\pm}$}
For definiteness, set
\begin{align}\label{eq:Vpm}
\bV_+&:=e^{\bsigma^+}
\left[ \begin {array}{cc} 2 \sinh \left( \lambda \bbp \right) \lambda^{-1}\bbb\bbc
& 2 \sqrt {\gamma}\cosh \left( \lambda \bbp \right)\bbc 
+\im \sinh \left( \lambda \bbp \right) \lambda^{-1}\bba\bbc
\\\noalign{\medskip}
 -2 \sqrt {\gamma}\cosh \left( \lambda \bbp \right) \bbb
+\im\sinh \left( \lambda \bbp \right)\lambda^{-1}  \bba\bbb
&-2 \sinh \left(  \lambda \bbp \right) \lambda^{-1}\bbb\bbc \end {array}
 \right] , \\
\bV_-&:=e^{\bsigma^+}
\left[ \begin {array}{cc} -2 \cosh \left( \lambda \bbp \right)  \bbb\bbc
& -2 \sqrt {\gamma}\sinh \left( \lambda \bbp \right)\lambda \bbc 
-\im \cosh \left(  \lambda \bbp \right)  \bba\bbc
\\\noalign{\medskip}
 2 \sqrt {\gamma}\sinh \left( \lambda \bbp \right) \lambda\bbb
-\im\cosh \left( \lambda \bbp \right)  \bba\bbb
&2 \cosh \left( \lambda \bbp \right)  \bbb\bbc \end {array}
 \right]. \n
\end{align}
%They satisfy the Killing field equation 
%$$\ed\bV_{\pm}+[\cbphi,\bV_{\pm}]=0.$$

Summarizing the analysis so far, 
the Killing fields for $\cbphi$ are determined as follows.
%Note that $\bV_+$ is $\bmg_{\geq 0}$-valued, and $\bY, \bV_-$ are $\bmg_{\geq 1}$ valued.
\begin{thm}\label{thm:VV}
Let $\bV_{\pm}$ be  given by \eqref{eq:Vpm}. Then
they satisfy the Killing field equation for $\cbphi$,
$$\ed\bV_{\pm}+[\cbphi,\bV_{\pm}]=0.$$
The set of three Killing fields $\{\, \bY, \bV_{\pm}\}$ generates the space of
$\bmg_{\geq 0}$-valued Killing fields for $\cbphi.$
\end{thm}
Note that $\bV_+$ is $\bmg_{\geq 0}$-valued, and $\bY, \bV_-$ are $\bmg_{\geq 1}$ valued.
 
%-------------------------------------------------------------- 
\subb{Algebraic identities}
Note the following Lie-bracket relations,
\begin{align}\label{eq:YVV}
[\bY,\bV_+]&= 4\sqrt{\gamma}  \bV_-, \qquad
[\bY,\bV_-]= 4\sqrt{\gamma}\lambda^2 \bV_+, \\
&\quad[\bV_+,\bV_-]=( 4 e^{2\bsigma^+}\bbb\bbc)\sqrt{\gamma} \bY=\sqrt{\gamma} \bY.\n
\end{align}%\marg{check}%%%%%%%%%%%%%%%%%%%%
In fact, these are implied by the following equations,
\begin{align}\label{eq:YVYV'}
\bY \bV_+= -\bV_+\bY&= 2\sqrt{\gamma}  \bV_-, \qquad
\bY \bV_- =-\bV_-\bY= 2\sqrt{\gamma}\lambda^2 \bV_+, \\
&\quad \bV_+ \bV_- =-  \bV_- \bV_+=\frac{1}{2}\sqrt{\gamma} \bY.\n
\end{align}
 
Note the determinant formulas,
\begin{align}\label{eq:VVdet}
\det(\bV_+)&= \gamma, \qquad\qquad
\det(\bV_-)= -\gamma \lambda^2, \\
&\det(\bY)=\bba^2-4\bbb\bbc=-4\gamma\lambda^2.\n
\end{align}
Note also the trace relations $$\tr(\bY\bV_{\pm}),\; \tr(\bV_+\bV_-)=0.$$
%----------------------------------------------------------------------------
%----------------------------------------------------------------------------
\subsection{Killing fields for $\pmb{\phi}$ }%$\tn{$\bphi$}$}
A generating set of the (formal) solutions to the Killing field equation for $\bphi$
can be obtained from $\{ \bY, \bV_{\pm}\}$ by a linear transformation.

\two
Recall
$$\ttt=\frac{1}{2\im}\sum_{m=0}^{\infty}\lambda^{-(2m+2)}t_m =\int\balpha.$$
Set
\be\label{eq:PPpm}
\bp \bP_+ \\\bP_-\ep:=
\bp  
\cosh(4\sqrt {\gamma}\lambda\ttt)  &  -  \sinh(4\sqrt {\gamma}\lambda\ttt)\lambda^{-1}     \\
 -\sinh(4\sqrt {\gamma}\lambda\ttt)\lambda& \cosh(4\sqrt {\gamma}\lambda\ttt)       \\
\ep
\bp \bV_+ \\ \bV_-\ep.
\ee
\iffalse
\begin{align}\label{eq:PPpm}
\bP_+&=  \cosh(4\sqrt {\gamma}\lambda\ttt) \bV_+
-  \sinh(4\sqrt {\gamma}\lambda\ttt)\lambda^{-1}\bV_-, \\
\bP_- &=-\sinh(4\sqrt {\gamma}\lambda\ttt)\lambda \bV_+   
+ \cosh(4\sqrt {\gamma}\lambda\ttt) \bV_-.\n
\end{align}\fi
A direct computation shows that they satisfy the Killing field equation for $\bphi$.
%$$\ed \bP_{\pm}+[\bphi, \bP_{\pm}]=0.$$
\begin{thm}\label{thm:PP}
Let $\bP_{\pm}$ be  given by \eqref{eq:PPpm}. 
Then they satisfy  (formally) the Killing field equation for $\bphi$,
$$\ed\bP_{\pm}+[\bphi,\bP_{\pm}]=0.$$
The set of three Killing fields $\{\, \bY, \bP_{\pm}\}$ generates 
the space of (formal) Killing fields for $\bphi$ which are of the form 
\be\label{eq:oftheform}
\bQ_0+ \cosh(4\sqrt {\gamma}\lambda\ttt)\bQ_+
+\sinh(4\sqrt {\gamma}\lambda\ttt)\lambda\bQ_-,\ee 
where $\bQ_0, \bQ_{\pm}$ are $\bmg_{\geq 0}$-valued.
\end{thm}
We omit the details of proof.
%-------------------------------------------------------------- 
\subb{Algebraic identities}
The Lie-bracket relations among $\{\, \bY, \bP_{\pm}\}$ are identical to \eqref{eq:YVV},
\begin{align}\label{eq:YPP}
[\bY,\bP_+]&= 4\sqrt{\gamma}  \bP_-, \qquad
[\bY,\bP_-]= 4\sqrt{\gamma}\lambda^2 \bP_+, \\
&\qquad[\bP_+,\bP_-]%=( 4 e^{2\bsigma^+}\bbb\bbc)\sqrt{\gamma} \bY
=\sqrt{\gamma} \bY.\n 
\end{align}
We also have,
\begin{align}\label{eq:YPYP'}
\bY \bP_+= -\bP_+\bY&= 2\sqrt{\gamma}  \bP_-, \qquad
\bY \bP_- =-\bP_-\bY= 2\sqrt{\gamma}\lambda^2 \bP_+, \\
&\quad \bP_+ \bP_- =-  \bP_- \bP_+=\frac{1}{2}\sqrt{\gamma} \bY.\n
\end{align}

Note the determinant formulas,
\begin{align}\label{eq:PPdet}
\det(\bP_+)&= \gamma, \qquad\qquad
\det(\bP_-)= -\gamma \lambda^2, \\
&\det(\bY)%=\bba^2-4\bbb\bbc
=-4\gamma\lambda^2.\n
\end{align}
Note also the trace relations $$\tr(\bY\bP_{\pm}),\; \tr(\bP_+\bP_-)=0.$$

%\np
%%%%%%%%%%%%%%
%%%%%%%%%%%%%%
%%%%%%%%%%%%%%
%%%%%%%%%%%%%%
%%%%%%%%%%%%%%%%%%%%%%%%%%%%
\section{Spectral Killing field}\label{sec:spectral}
Due to the terms $\cosh(4\sqrt {\gamma}\lambda\ttt), \sinh(4\sqrt {\gamma}\lambda\ttt)$,
the Killing fields $\bP_{\pm}$ for $\bphi$ are formally defined objects and
it is not obvious how to draw a geometrically meaningful conclusion from them.

In this section, we introduce the non-local,  quasi-Killing fields for $\bphi$
called   \emph{spectral Killing fields},
by utilizing the weighted homogeneous property of the CMC system
with respect to the spectral parameter $\lambda$.
A spectral Killing field, denoted by $\bmcs$, is defined by 
the inhomogeneous Killing field equation \eqref{eq:inhomKF} given below.

Unlike  $\bP_{\pm}$, on which the construction of $\bmcs$ will be based,
it will be shown that a suitably normalized spectral Killing field, defined by \eqref{eq:Rformula}, 
is a  well defined $\bmg$-valued function;
when expanded as a formal Laurent series in $\lambda$,
it admits a finite expression for each coefficient.  
 
\two
Let us  introduce a relevant notation. 
Let $$\mcd:=\mcl_{\lambda \dd{}{\lambda}}$$
be the Euler operator with respect to the spectral parameter $\lambda$.
For a scalar function, or a differential form $A$, the notation  $\bdot{A}$ (upper-dot) would mean
the application of the Euler operator,
$$\bdot{A}=\mcd (A).$$

%-----------------------------------------------------------------------
%-----------------------------------------------------------------------
\sub{Inhomogeneous Killing field equation}
Consider the inhomogeneous Killing field equation 
\be\label{eq:inhomKF}
\ed \bmcs +[\bphi,\bmcs]=\bdot{\bphi}.
\ee
For the sake of notation, let $\Fh{\infty}_+$ denote the infinite jet space of the CMC hierarchy.\ftmark\fttext{Recall from Part \Rmnum{1} that 
 $\Fh{\infty}_+=\Fh{\infty} \times \{\tba_{n}\}_{n=0}^{\infty}\times \{ t_m\}_{m=0}^{\infty}.$ }
\begin{defn}
A \tb{spectral Killing field} is a $\bmg$-valued function 
$$\bmcs: \Fh{\infty}_+\to\bmg,$$
which satisfies the inhomogeneous Killing field equation \eqref{eq:inhomKF}.  
By definition, the space of spectral Killing fields is an affine space over 
the space of Killing fields for $\bphi$.
\end{defn}
The term spectral Killing field, although such $\bmcs$ is not exactly a Killing field, is justified 
by the following observation;
since the $\lambda$-degree of  $\bphi$, and hence of $\bdot{\bphi}$, is bounded from above
by $2N+1$,
the components of $\bmcs$ of $\lambda$-degree$\,\geq 2N+2$
do satisfy a version of the recursive structure equations \eqref{eq:abcstrt} 
for the ordinary Killing fields.
\begin{rem}\label{rem:compare}
Compare this definition of the \emph{higher-order} spectral Killing fields 
with that of the \emph{classical} spectral symmetries adopted in \cite[\S12]{Wang2013}.
The space of classical  spectral symmetries
 is by definition an affine space over the space of classical symmetries.
\end{rem}

\one
It follows that, by a similar argument as for the formal Killing field $\bY$,
a spectral Killing field  contains an infinite sequence of 
the corresponding  spectral Jacobi fields. % and conservation laws.
 
%\begin{rem}\label{rem:physics}
%In the physics literature, spectral symmetry   is often called 
%\emph{Virasoro symmetry}, \cite[Sec.4]{Nicolai1991b}.
%\end{rem}

On the other hand, the presence of the intermediate part of $\bmcs$ 
(of $\lambda$-degree$\,\leq 2N+1$)
will later re-emerge and impose the constraints on the Virasoro algebras on which 
the further extension of the CMC hierarchy will be based;
it requires that the corresponding Virasoro algebras are   truncated from below, \S\ref{sec:hVir}.
%-----------------------------------------------------------------------
\subb{Motivation}
The basis for introducing Eq.\eqref{eq:inhomKF} is the following observation.
\begin{lem}\label{lem:basic}
Consider the inhomogeneous Killing field equation  
for a given 1-form $\bomega$:
\be 
\ed\bmcs+[\bphi,\bmcs]=\bomega.\n
\ee
Then this  equation for $\bmcs$ is compatible whenever  
$$\ed\bomega+[\bphi,\bomega] =
\ed\bomega+\bphi\w\bomega+\bomega\w\bphi=0.$$
\end{lem}\begin{proof}Direct computation.
\end{proof}
For the proposed spectral Killing field equation ($\bomega=\bdot{\bphi}$), note that
$$\ed\bphi+\bphi\w\bphi=0 \quad\xrightarrow{\;\;\;\mcd\;\;\;}\quad \ed\bdot{\bphi}+\bphi\w\bdot{\bphi}+\bdot{\bphi}\w\bphi=0,
$$
and the compatibility condition is satisfied.
%-----------------------------------------------------------------------
%-----------------------------------------------------------------------
\sub{Spectral identities}
We record the characteristic properties of the spectral Killing fields. 
Let $\tn{D}$ denote the differential operator for the Killing field equation   for $\bphi$:
$$\tn{D}(\cdot):=\ed(\cdot)+[\bphi,(\cdot)].$$ 
Recall $\mcd=\lambda \dd{}{\lambda}$ is the Euler operator.
\begin{prop}\label{prop:Uabc}
Let $\bmcs$ be a spectral Killing field. 
\begin{enumerate}[\qquad a)]
\item
Let $\bP$ be a  Killing field for $\bphi$.
Consider the Lie bracket $[\bP,\bmcs]$.  Then
$$\tn{D}([\bP,\bmcs]) =[\bP,\bdot{\bphi}].
$$
\item
Note that
\be\label{eq:dPdot}
\ed\bP+[\bphi,\bP]=0  \quad\xrightarrow{\;\;\;\mcd\;\;\;}\quad
\tn{D}(\bdot{\bP})=\ed\bdot{\bP}+[\bphi,\bdot{\bP}]=[\bP,\bdot{\bphi}].
\ee
%Hence, up to adding formal Killing fields $\{ \bY, \bP_{\pm}\}$ to $\bmcs$, one may normalize so that
By Thm.\ref{thm:PP},  this implies,
\be\label{eq:YPPR}
[\bY,\bmcs]\equiv\bdot{\bY}, \qquad [\bP_{\pm},\bmcs]\equiv\bdot{\bP}_{\pm}\mod \bY,\bP_{\pm}\footnotemark.
\footnotetext{Here ``mod $ \bY,\bP_{\pm}$" means modulo the vector space of  Killing fields generated by $ \bY,\bP_{\pm}$.}
\ee
%This condition uniquely determines $\bmcs$.
%(We shall  impose this normalization from now on.)
\item
Conversely, suppose a $\bmg$-valued function $\bmcs$ satisfies the three algebraic relations \eqref{eq:YPPR}. Then $\bmcs$ is a spectral Killing field.
\end{enumerate}
\end{prop}
\begin{proof}
a), b)
By Thm.\ref{thm:PP}, the vector space of Killing fields for $\bphi$ is spanned by 
$\{ \bY, \bP_{\pm}\}$. The rest follows by a direct computation.

c)
Differentiate the algebraic relations, and Eq.\eqref{eq:dPdot} gives
$$[\bY, \tn{D}\bmcs-\bdot{\bphi}]=0, \qquad [\bP_{\pm}, \tn{D}\bmcs-\bdot{\bphi}]=0.$$
This forces $\tn{D}\bmcs=\bdot{\bphi}.$
\end{proof}

The identity \eqref{eq:dPdot} allows one to algebraically solve for 
a particular spectral Killing field.
\begin{prop}\label{prop:spectral}
Let $\bmcs$ be given by 
\be\label{eq:Rformula}
\bmcs:=c_0\ad_{\bY}(\bdot{\bY})+c_+\ad_{\bP_+}(\bdot{\bP}_+)+c_-\ad_{\bP_-}(\bdot{\bP}_-),
\ee
where
\be 
c_0=+\frac{1}{32\gamma}\lambda^{-2},  \quad
c_{+}=- \frac{1}{32\gamma e^{2\bsigma^+}\bbb\bbc}=-\frac{1}{8\gamma},
\quad
c_{-}=+ \frac{1}{32\gamma e^{2\bsigma^+}\bbb\bbc}\lambda^{-2}=\frac{1}{8\gamma}\lambda^{-2}. \n
\ee
Then $\bmcs$ is a  spectral Killing field.
\end{prop}
\begin{proof}
We directly compute $\tn{D}\bmcs$ from the given formula. By \eqref{eq:dPdot},
$$\tn{D}\bmcs=c_0\ad_{\bY}^2(\bdot{\bphi})+c_+\ad_{\bP_+}^2(\bdot{\bphi})+c_-\ad_{\bP_-}^2(\bdot{\bphi}).
$$
The claim  $\tn{D}\bmcs=\bdot{\bphi}$ follows from the operator identity
$$c_0\ad_{\bY}^2+c_+\ad_{\bP_+}^2+c_-\ad_{\bP_-}^2=1_{\bmg}.
$$
%The normalization conditions \eqref{eq:YPPR} is checked by a direct computation using Jacobi identity. 
%We omit the details.
\end{proof}
\begin{defn}
The spectral Killing field defined by \eqref{eq:Rformula}
is  the \tb{normalized} spectral Killing field.
\end{defn}
This choice of terminology is based on the Lie bracket relations \eqref{eq:mus} given below.
From now on, we will only consider the normalized spectral Killing field.

%-----------------------------------------------------------------------
%-----------------------------------------------------------------------
\sub{Spectral Killing field for $\cbphi$}
The spectral Killing field $\bmcs$ admits an alternative formula in terms of $\{ \bY, \bV_{\pm}\}$,
which is suitable for application.

\two
Consider
\be\label{eq:Rcformula}
\check{\bmcs}:=c_0\ad_{\bY}(\bdot{\bY})+c_+\ad_{\bV_+}(\bdot{\bV}_+)+c_-\ad_{\bV_-}(\bdot{\bV}_-).
\ee
Note that, since $\bY, \bV_-$ are $\bmg_{\geq 1}$-valued and
$\bV_+$ is $\bmg_{\geq 0}$-valued, $\check{\bmcs}$ is $\bmg_{\geq 1}$-valued.

By the similar argument as before, we have the equations  
\begin{align}\label{eq:YVVR}
&\ed\check{\bmcs}+[\cbphi,\check{\bmcs}]=\bdot{\cbphi},\\
&[\bY,\check{\bmcs}]\equiv\bdot{\bY}, \qquad [\bV_{\pm},\check{\bmcs}]\equiv\bdot{\bV}_{\pm}
\mod \bY, \bV_{\pm}.\n
\end{align}
In fact, we have the following exact formulas.
\begin{lem}\label{lem:ScYVcommute} 
Recall the normalization $ 4 e^{\bsigma^+}\bbb\bbc =1.$ Under this condition, 
\begin{align}\label{eq:ScYVcommute}
[\bV_+,\check{\bmcs}]-\bdot{\bV}_+&=0,\qquad [\bV_-,\check{\bmcs}]-\bdot{\bV}_- =-\bV_-,\\
&[\bY,\check{\bmcs}]-\bdot{\bY}=-\bY.\n
\end{align}
\end{lem}
\begin{proof}
Applying the Jacobi identity to $[[\bY,\bV_{\pm}],\check{\bmcs}], [[\bV_+,\bV_-],\check{\bmcs}]$,
one gets six linearly independent relations.
Eq.\eqref{eq:Rcformula} gives the remaining three relations.
\end{proof}
%The formula
%\begin{center}\label{eq:SYformula}\fbox{$[\check{\bmcs},\bY]+(\bdot{\bY}-\bY)=0$}\end{center}

%%%%%%%%%%%%%%%%
%%%%%%%%%%%%%%%%%%%%%%%%%%%
\sub{Spectral Killing field for $\pmb{\phi}$}
By definition, 
\begin{align}
\bdot{\cbphi}&=\ed \check{\bmcs}+[\cbphi,\check{\bmcs}] \n\\
&=\ed \check{\bmcs}+[\bphi-\bY\balpha,\check{\bmcs}] \n\\
&=\ed \check{\bmcs}+[\bphi,\check{\bmcs}]-[\bY,\check{\bmcs}]\balpha.\n
\end{align}
Hence, by \eqref{eq:ScYVcommute}, we have
\begin{align}
\tn{D}\check{\bmcs}&=\bdot{\cbphi}+[\bY, \check{\bmcs}]\balpha \n\\
&=\bdot{\cbphi}+(\bdot{\bY}-\bY)\balpha. \n
\end{align}
Since
$$\bY\balpha=\bphi-\cbphi\quad\xrightarrow{\;\;\;\mcd\;\;\;}\quad
\bdot{(\bY\balpha)}=\bdot{\bphi}-\bdot{\cbphi},$$
after substitution one gets
\begin{align}
\tn{D}\check{\bmcs}&=\bdot{\cbphi}+(\bdot{\bY}-\bY)\balpha \n\\
&=\bdot{\bphi}-(\bdot{\bY}\balpha+\bY\bdot{\balpha})+(\bdot{\bY}-\bY)\balpha\n\\
&=\bdot{\bphi}-(\bdot{\balpha}+ \balpha)\bY. \n  
\end{align}
It follows that
$$\bmcs\equiv\left(\int\left(\bdot{\balpha}+\balpha\right)\right)\bY
+\check{\bmcs},  \mod\bY,\bP_{\pm}.
$$
\begin{thm}\label{thm:spectralKilling}
Suppose the normalization of $\bsigma^+$ by
$4 e^{\bsigma^+}\bbb\bbc=1.$
Then the spectral Killing field \eqref{eq:Rformula} can be written in terms of 
$\{ \bY, \bV_{\pm}\}$ as
\begin{align}\label{eq:RVformula}
\bmcs&=\left(\bdot{\ttt}+\ttt \right)\bY+\check{\bmcs}.
% +\frac{1}{4\sqrt{\gamma}}\left(\mu_-\bV_++\mu_+\lambda^{-2}\bV_-\right)  
% \mod \bY, \bP_{\pm}.  
\end{align}
Here the scalar function $\bdot{\ttt}$ is given by
$\bdot{\ttt}=\int\bdot{\balpha}=\frac{1}{2\im}\sum_{m=0}^{\infty}(-2m-2)\lambda^{-(2m+2)}t_m,$
and hence
$$\bdot{\ttt}+\ttt=\frac{\im}{2}\sum_{m=0}^{\infty}(2m+1)\lambda^{-(2m+2)}t_m.$$
Note that $\bmcs$ is $\bmg$-valued.
\iffalse
When $t_N$-truncated for some $N\geq 0$, 
the $\lambda$-degree of the term $\bdot{\ttt}+\ttt$ is bounded from below. In this case, the $\lambda$-degree of  $\bmcs$ is also bounded from below
and, when expanded as a series in $\lambda, \lambda^{-1}$, each coefficient of $\bmcs$ becomes a finite expression.\fi
\end{thm}
\begin{proof}
Substitute \eqref{eq:PPpm} to \eqref{eq:Rformula}.
\end{proof}
%--------------------------------------------------------
\subb{Spectral identities}
We record the analogue of  Lem.~\ref{lem:ScYVcommute} for $\{\bY, \bP_{\pm}\}$.
\begin{lem}\label{lem:SYPcommute} 
Recall the normalization $ 4 e^{\bsigma^+}\bbb\bbc =1.$ Under this condition, 
\begin{align}\label{eq:mus}
[\bP_+, {\bmcs}]-\bdot{\bP}_+&=0,\qquad [\bP_-, {\bmcs}]-\bdot{\bP}_- =-\bP_-,\\
&[\bY, {\bmcs}]-\bdot{\bY}=-\bY.\n
\end{align}
\end{lem}
The  last equation for $[\bY,\bmcs]$ can be rewritten as; 
\be\label{eq:keyformula}
\fbox{\quad$[ {\bmcs},\lambda^{-1}\bY]+ \bdot{(\lambda^{-1} \bY)}=0.$\quad} 
\ee
This formula will be important for the further extension of the Maurer-Cartan form $\bphi$
to the associated affine Kac-Moody algebra valued 1-form,
in order to incorporate the  spectral symmetries
for  the construction of the extended CMC hierarchy.

%--------------------------------------------------------
\subb{$\det{\bmcs}$}
Note from \eqref{eq:inhomKF} that
$$\ed (\bmcs^2) =(\ed\bmcs)\bmcs+\bmcs(\ed\bmcs)=
\bdot{\bphi}\bmcs+\bmcs\bdot{\bphi}.
$$
Here we used the identity $\bmcs^2+\det(\bmcs)I_2=0$, for $\tr(\bmcs)=0$. 
Take the trace, and one finds
\be\label{eq:ddetS}
\ed(\det\bmcs)=-\tr(\bmcs\bdot{\bphi}).
\ee

%\np
%%%%%%%%%%%%%
%%%%%%%%%%%%% 
%%%%%%%%%%%%% 
%%%%%%%%%%%%% 
%%%%%%%%%%%%% %%%%%%%%%%%%% 
\section{Affine Killing fields}\label{sec:affineKilling}
In Part \Rmnum{1},
the CMC hierarchy was obtained by extending the original CMC system
by the higher-order commuting symmetries
generated by the formal Killing field $\bY$.
We wish  to define a further extension of the CMC hierarchy
by the  (non-commuting) symmetries generated by the spectral Killing field $\bmcs.$

We will find that the differential algebraic relations among the ingredients of the construction
are consistent with the Lie algebra structure of the associated generalized affine Kac-Moody algebras,
denoted by  $\hbmg^{\pm}_{N+1}$; 
subsequently, they will be packaged into the $\hbmg^{\pm}_{N+1}$-valued extended Maurer-Cartan forms
and the corresponding $\hbmg^{\pm}_{N+1}$-valued\ftmark\fttext{Up 
to scaling the loop algebra part by $\lambda^{-1}$. See \S\ref{sec:CMC+}.} %%%%%%%%%%%%%
affine Killing fields.  

To this end, we give a definition of the generalized affine Kac-Moody algebras $\hbmg^{\pm}_{N+1}$
as the  loop algebra $\bmg$ enhanced by the truncated Virasoro algebras  of even derivations,
Defn.\ref{defn:truncatedVir}, Defn.\ref{defn:affineKM}.
For an indication of  the later construction,
we then show that the triple of data $(\bphi, \bY,\bmcs)$   
admit a lift to $\hbmg^{\pm}_{N+1}$-valued affine Killing field equations.
%\marg{after all the extension, ornamental}%%%%%%%%%%%%
%--------------------------------------------------------
%--------------------------------------------------------
\sub{Affine Kac-Moody algebras}
From the classification of Kac-Moody algebras, 
an affine Kac-Moody algebra is the central extension of a (twisted) loop algebra enhanced
by a derivation, \cite{Kac1990}.
For our purpose, the relevant Lie algebra is the semi-direct product of the loop algebra $\bmg$
with  a truncated Virasoro algebra of even derivations,  \S\ref{sec:hVir}.
We shall abuse the terminology and call this an affine Kac-Moody algebra.

Since the central component does not contribute to nor affect the construction 
of the extended CMC hierarchy,
it is postponed to \S\ref{sec:central}.
%--------------------------------------------------------
\subb{Truncated (centerless) Virasoro algebras}\label{sec:hVir}
Let the Euler operator
$$\mcd=\lambda\dd{}{\lambda}$$
now be considered as  the operator  acting on $\sla(2,\C)[[\lambda^{-1},\lambda]]$
as a derivation. 
Under the formal complex conjugation ($\lambda\to\lambda^{-1}$), note that
$$\ol{\mcd}=\lambda^{-1}\dd{}{\lambda^{-1}}=\lambda^{-1}\dd{\lambda}{\lambda^{-1}}
\dd{}{\lambda}=-\mcd.$$

For integers $\ell, k \geq 0$, let
$$\delsb{\ell}=\tcr{-}\, \lambda^{2\ell} \mcd, \qquad
\dels{k}= \lambda^{-2k}\mcd.
$$
Note the commutation relations
\be\label{eq:ssbracket}
[\delsb{\ell},\delsb{s}] =  (2\ell-2s)\delsb{\ell+s},\qquad
[\dels{j}, \dels{k}]= (2j-2k)\dels{j+k}.% \n\\
%& [\delsb{\ell}, \dels{k}]=(2k-2\ell)\dels{}.
\ee
\begin{defn}\label{defn:truncatedVir}%%%%%%%%%%%%%%%
Let $N\geq 0$ be a non-negative integer.
The \tb{$N$-truncated (centerless) Virasoro algebras} are defined by
\be\label{eq:Virdefn}
\Vir^+_{N+1}:=\langle \delsb{\ell} \rangle_{\ell\geq N+1}, \qquad
\Vir^-_{N+1}:=\langle \dels{k} \rangle_{k\geq N+1}.
\ee
They are called of positive/negative types respectively.
\end{defn}

Let $\{ \sigmab_{\ell}\}, \{\sigma_k\}$ be the 1-forms formally dual to 
$\{\delsb{\ell}\}, \{\dels{k}\}$ respectively.
Consider the generating series
$$\bsigma_+:=\sum_{\ell=0}^{\infty} \tcr{-}\, \lambda^{2\ell}\sigmab_{\ell},  \quad
\bsigma_-:=\sum_{k=0}^{\infty} \lambda^{-2k}\sigma_k.$$
Given the truncation parameter $N$, it is understood here that:
$$
\fbox{ $\quad  \sigmab_{\ell},   \sigma_k =0,\;\;\forall \,\ell, k \leq N.\quad$
}$$
For a uniform treatment, we retain the lower bound for the summation to be 0 instead of $N+1$.

Then,  the structure equation for $\Vir^{\pm}_{N+1}$ is
written in terms of the generating 1-form by
\be\label{eq:Virstrt}
\ed\bsigma_{\pm}+\bsigma_{\pm}\w\bdot{\bsigma_{\pm}}=0.
\ee
Collecting the terms with respect to $\lambda$-degree, this gives the following formal
structure equations for $\sigmab_{\ell}, \sigma_k$:
\be\label{eq:sigmastrt}
\ed\sigmab_{i}=\sum_{\ell+s=i} (s-\ell)\sigmab_{\ell}\w\sigmab_s,
\qquad
\ed\sigma_i=\sum_{j+k=i} (k-j)\sigma_j\w\sigma_k.
\ee
Note that they agree  with \eqref{eq:ssbracket}.

%--------------------------------------------------------
\subb{Affine Kac-Moody algebras}
The Lie algebras $\Vir^{\pm}_{N+1}$  naturally act on $\bmg$ as derivations.
\begin{defn}\label{defn:affineKM}%%%%%%%%%%%%%
Let $\bmg$ be the twisted loop algebra \eqref{eq:twisted}.
Given the truncation parameter $N\geq 0$,
the associated (centerless) \tb{affine Kac-Moody algebras} $\hbmg^{\pm}_{N+1}$ are defined
as the semidirect product
\be 
\hbmg^{\pm}_{N+1}:=  \Vir^{\pm}_{N+1}  \ltimes \bmg. \n
\ee
They are called of positive/negative types respectively.
\end{defn}

%--------------------------------------------------------
%\subb{Central extension}

%--------------------------------------------------------
%--------------------------------------------------------
\sub{Affine lifts of Killing fields}\label{sec:affinelift}
The formal Killing field $\bY$ and the spectral Killing field $\bmcs$
admit  the \emph{affine lifts} to the $\hbmg^{\pm}_{N+1}$-valued Killing fields as follows.

\two
Set
$$\hat{\bY}=(0,\bY),\quad \hat{\bmcs}=(\tne^{u_{\pm}},\tne^{u_{\pm}}\bmcs)$$
be the lifts of $\{\bY,\bmcs \}$ respectively to the $\hbmg^{\pm}_{N+1}$- valued functions,
where
$$\tne^{u_{\pm}}\in \C[[\lambda^{\pm 2}]]\lambda^{\pm(2N+2)}$$
(the first component is the derivation part).
Set
$$\bPhi=(0,\bphi)$$
be the trivial lift of $\bphi$ to the $\hbmg^{\pm}_{N+1}$-valued 1-form.
Then it is easily checked that they satisfy the corresponding 
$\hbmg^{\pm}_{N+1}$-valued Killing field equations;
\begin{align}\label{eq:lifteq}
\ed \hbY +[\bPhi,  \hbY]&=0,    \\
\ed \hbmcs +[\bPhi, \hbmcs ]&=0,  \n  \\
\ed\bPhi+\frac{1}{2}[\bPhi,\bPhi]&=0. \n
\end{align}
Note here the loop algebra component of the second equation gives
$$\ed (\tne^{u_{\pm}}\bmcs)+[\bphi,\tne^{u_{\pm}}\bmcs]-\tne^{u_{\pm}}\bdot{\bphi}=0.
$$
Up to scaling by $\tne^{u_{\pm}}$, 
this agrees with Eq.\eqref{eq:inhomKF}.

%\np
%%%%%%%%%
%%%%%%%%% 
%%%%%%%%% 
%%%%%%%%% 
%%%%%%%%%%%%%%%%%%  
\section{Extension by non-commuting symmetries}\label{sec:extension}
As a first step toward the construction of the extended CMC hierarchy,
we examine an extension of the structure equation for $\bY$
by the symmetries generated by $\bmcs$.
The extension formulas are dictated by the commutation relation \eqref{eq:keyformula},  
\S\ref{sec:motivationextY}.
We check the partial compatibility and
show that the additional equations commute with the original equations of the CMC hierarchy.

In view of the Lie algebra structure of $\hbmg^{\pm}_{N+1}$,
the extended structure equation for $\bY$  is  suggestive
of the corresponding extension of the Maurer-Cartan form $\bphi$.
The full description of the extended CMC hierarchy will be given in 
\S\S\ref{sec:CMC+}-\ref{sec:exformulas}.

%--------------------------------------------------------
%--------------------------------------------------------
\sub{Extension for $\bY$}\label{sec:extendY}
%--------------------------------------------------------
\subb{Motivation}\label{sec:motivationextY}
Consider the decomposition
$$\lambda^{-2k}\bmcs:=\mcs_k+\mcs_{(k+1)}\in \bmg_{\leq -1}+^{vs} \bmg_{\geq 0}.$$
From \eqref{eq:keyformula},
$$[\mcs_k+\mcs_{(k+1)}, \lambda^{-1}\bY]+\lambda^{-2k}\bdot{(\lambda^{-1}\bY)}=0.$$
Hence
\be\label{eq:pmrelationS}
- [\mcs_k, \lambda^{-1}\bY]-\lambda^{-2k}\bdot{(\lambda^{-1}\bY)}
=[\mcs_{(k+1)},\lambda^{-1}\bY]\in\bmg_{\geq 0},
\ee
and the LHS of this equation has no $\bmg_{\leq -1}$-part.
Equivalently,
\be\label{eq:pmrelationS'}
- [\mcs_k,  \bY]-\lambda^{-2k} ( \bdot{\bY}-\bY) 
=[\mcs_{(k+1)}, \bY]\in\bmg_{\geq 1}.
\ee

%--------------------------------------------------------
\subb{Extension for $\bY$}
This observation suggests the following partial extension of the structure equation for $\bY$.

\two
Recall the basis  $\{  \delsb{\ell} \}_{\ell\geq N+1}, \{ \dels{k}  \}_{k\geq N+1}$ 
of the truncated Virasoro algebras $\Vir^{\pm}_{N+1}$ respectively, \S\ref{sec:hVir}.
Consider the following extended structure equation for $\bY$:

\two
(for $0\leq m,n\leq N, \;  N+1\leq  k, \ell$) %\ftmark\fttext{In fact, the part of commutativity discussed here holds in general for $m,n,k, \ell \geq 0.$}
\be\label{eq:Y+}\left\{
\begin{array}{rlrl}
\del_{t_m}\bY&=-[U_m, \bY], &  \del_{t_m}\ol{\bY}^t&=-[U_m,\ol{\bY}^t],                        \\
\del_{\tba_n }\bY&=+[\ol{U}^t_n,\bY],& \del_{\tba_n }\ol{\bY}^t&=+[\ol{U}^t_n,\ol{\bY}^t],\\
\dels{k}\bY&=-[\mcs_k,\bY]-\lambda^{-2k}(\bdot{\bY}-\bY),&
\dels{k}\ol{\bY}^t&=-[\mcs_k,\ol{\bY}^t]\,\tcr{-}\, \lambda^{-2k}(\bdot{\ol{\bY}^t}+\ol{\bY}^t), \\
\delsb{\ell}\bY&=+[\ol{\mcs}^t_{\ell},\bY] \,\tcr{+} \,\lambda^{+2\ell}(\bdot{\bY}-\bY),& 
\delsb{\ell}\ol{\bY}^t&=+[\ol{\mcs}^t_{\ell},\ol{\bY}^t] 
\,\tcr{+}\,\lambda^{+2\ell}(\bdot{\ol{\bY}^t}+\ol{\bY}^t).
\end{array}\right.
\ee

\one
We claim that;
\beit
\item the $\delsb{\ell}, \dels{k}$-flows for $\bY$  commute
with  $\del_{\tba_n}, \del_{t_m}$-flows respectively,
$$[\delsb{\ell}, \del_{\tba_n}]\bY=0,\quad [\dels{k}, \del_{t_m}]\bY=0.$$
\item the differential operators $\delsb{\ell}, \dels{k}$ respectively satisfy the Virasoro relations,
 $$[\delsb{\ell},\delsb{s}]\bY=(2\ell-2s)\delsb{\ell+s}\bY,
\quad [\dels{j},\dels{k}]\bY=(2j-2k)\dels{j+k}\bY.$$
\enit

In this section, we verify only the first claim. This shows that 
the proposed system of  equations \eqref{eq:Y+}
%of  $\delsb{\ell}, \dels{k}$-flows for $\bY$ respectively
is indeed a (part of) symmetry extension of the CMC hierarchy.
For the second claim, the $\delsb{\ell}, \dels{k}$-derivatives of $\bmcs$ need to be introduced,
and this will be checked in the next section.
%--------------------------------------------------------
%--------------------------------------------------------
\sub{Compatibility}
We show that $$[\del_{t_m},\dels{k}]\bY=0.$$
The relation $[\del_{\tba_n},\delsb{\ell}]\conjt{\bY}=0$ follows by taking 
the  (obvious)  formal conjugate transpose.
%--------------------------------------------------------
\subb{Identities}\label{sec:identities}
Assume for the moment that $k, \ell  \geq 0$ are non-negative integers.
\begin{lem}
\begin{align}\label{eq:delSk}
\del_{t_m} \mcs_k&=-[U_m,\mcs_k]-[U_m,\mcs_{(k+1)}]_{\leq -1}
+(\lambda^{-2k}\bdot{U}_m)_{\leq -1},\\
\del_{\tbar{n}} \mcs_k&=+[\Uba^t_n,\mcs_k]_{\leq -1}-(\lambda^{-2k}\bdot{\Uba^t_n})_{\leq -1}.\n
\end{align}
\end{lem}
\begin{proof}
From the spectral Killing field equation
$$\ed\bmcs+[\bphi,\bmcs]=\bdot{\bphi},$$
we have
\begin{align}
\del_{t_m}  (\mcs_k+\mcs_{(k+1)})&=-[U_m,\mcs_k+\mcs_{(k+1)}]+\lambda^{-2k}\bdot{U}_m,\n\\
\del_{\tbar{n}}  (\mcs_k+\mcs_{(k+1)})&=+[\Uba^t_n,\mcs_k+\mcs_{(k+1)}]
-\lambda^{-2k}\bdot{\Uba^t_n}.\n
\end{align}
Collect the $\bmg_{\leq -1}$-terms.
\end{proof}
%%%%%%%%%%%%%%%%%%%%%%%%%
\begin{lem}
\begin{align}\label{eq:delUm}
\dels{k}U_m&=-[\mcs_k, U_m]-[\mcs_k, U_{(m+1)}]_{\leq -1}
-(2m+2k+1)U_{k+m}-\bdot{U}_{k+m},\\
\delsb{\ell}U_m&=+[\ol{\mcs}^t_{\ell},U_m]_{\leq -1}
+\left(\lambda^{2\ell}((2m+1)U_m+ \bdot{U}_m)\right)_{\leq -1}  \n  \\
&=+[\ol{\mcs}^t_{\ell},U_m]_{\leq -1}
+(2m-2\ell+1)U_{m-\ell}+\bdot{U}_{m-\ell}. \n
\end{align}
\end{lem}
\begin{proof}
From the third equation of \eqref{eq:Y+}, substitute $\bY=2\im\lambda^{2m+2}(U_m+U_{(m+1)})$
for the first two $\bY$'s, and  $\bY=2\im\lambda^{2k+2m+2}(U_{k+m}+U_{(k+m+1)})$ 
for the next $\bdot{\bY}, \bY$.
We have
$$\dels{k} (U_m+U_{(m+1)})=-[\mcs_k, U_m+U_{(m+1)}]
-(2m+2k+1)(U_{k+m}+U_{(k+m+1)})-(\bdot{U}_{k+m}+\bdot{U}_{(k+m+1)}).$$
Collect the $\bmg_{\leq -1}$-terms for $\dels{k} U_m.$
The formula for $\delsb{\ell}U_m$ is obtained in a similar way.
\end{proof}
%%%%%%%%%%%%%%%%%%%%%%%%
\begin{lem}
\be\label{eq:SkUm}
[\mcs_k,U_m]+\left([\mcs_k,U_{(m+1)}]+[\mcs_{(k+1)},U_m]\right)_{\leq -1}
=-(2k+2m+1)U_{k+m}-\bdot{U}_{k+m}.
\ee
\end{lem}
\begin{proof}
From the spectral identity $[\bmcs,\bY]=\bY-\bdot{\bY}$,
we have
$$[\mcs_k+\mcs_{(k+1)}, U_m+U_{(m+1)}]
=\frac{1}{2\im}\lambda^{-(2k+2m+2)}(\bY-\bdot{\bY}).$$
Substitute $\bY=2\im\lambda^{2k+2m+2}(U_{k+m}+U_{(k+m+1)})$, then the RHS becomes
$$\tn{RHS}=-(\bdot{U}_{k+m}+\bdot{U}_{(k+m+1)})
-(2k+2m+1)(U_{k+m}+U_{(k+m+1)}).$$
Collect the $\bmg_{\leq -1}$-terms.
\end{proof}

%------------------------------------
\subb{$[\del_{t_m},\dels{k}]\bY=0$}\label{sec:tmskY}
We compute $\dels{k}\del_{t_m}\bY$, and  $\del_{t_m}\dels{k}\bY$ in turn
using the identities \eqref{eq:delSk}, \eqref{eq:delUm}. 
Then,
\begin{align}\label{eq:sktm}
\dels{k}\del_{t_m}\bY&=\left[ 
[\mcs_k,U_m]+[\mcs_k,U_{(m+1)}]_{\leq -1}+(2m+2k+1)U_{k+m}+\bdot{U}_{k+m},
\bY \right] \\
&\quad+\left[ U_m, [\mcs_k,\bY]+\lambda^{-2k}(\bdot{\bY}-\bY)\right].\n
\end{align}
And,
\begin{align}\label{eq:tmsk}
\del_{t_m}\dels{k}\bY&=\left[ 
[U_m,\mcs_k]+[U_m, \mcs_{(k+1)}]_{\leq -1}-(\lambda^{-2k} \bdot{U}_m)_{\leq -1},
\bY \right]\\
&\quad+\left[\mcs_k, [U_m,\bY]\right]\n\\
&\quad-\lambda^{-2k}[U_m,\bY]+\lambda^{-2k}\left([\bdot{U}_m,\bY]+[U_m,\bdot{\bY}]\right).\n
\end{align}
Take the difference \eqref{eq:tmsk}$-$\eqref{eq:sktm} using \eqref{eq:SkUm},
and one gets
$$[\del_{t_m},\dels{k}]\bY=[(\lambda^{-2k} \bdot{U}_m)_{\geq 0},\bY].$$
Since $k\geq 0$ and $\bdot{U}_m$ is $\bmg_{\leq -1}$-valued,
$(\lambda^{-2k} \bdot{U}_m)_{\geq 0}=0$ and the claims follows.

%\np
%%%%%%%%%%%%
%%%%%%%%%%%%
%%%%%%%%%%%%
%%%%%%%%%%%%
%%%%%%%%%%%%%%%%%%%%%%%% 
\section{Extended CMC hierarchy}\label{sec:CMC+}
%Considering the structure equation for the affine Kac-Moody algebra,
The identity \eqref{eq:pmrelationS} is suggestive of
how to define the $\hbmg^{\pm}_{N+1}$-valued extension of $\bphi$ 
and  package the equations \eqref{eq:Y+}
into a $\hbmg^{\pm}_{N+1}$-valued Killing field equation.
We shall follow this idea and propose an extension of the CMC hierarchy
in terms of the  $\hbmg^{\pm}_{N+1}$-valued Killing field equations for 
the affine lifts $\hat{\bY}, \hbmcs$.
The $\sigmab_{l}, \sigma_{k}$-derivatives of $\bmcs$, the affine extension part, 
will follow from this (for free).

As remarked earlier, see below Rmk.\ref{rem:compare},
the derivation part of the underlying Lie algebra $\hbmg^{\pm}_{N+1}$
is the Virasoro algebra $\Vir^{\pm}_{N+1}$ 
which is truncated from below.
This is partly due to the fact that $\bmcs$ contains the intermediate portion
which does not satisfy the Killing field equation (inhomogeneous portion).
We will find analytically that 
this truncation of the Virasoro algebras is  dictated by the constraints
 from the compatibility equations for the extended CMC hierarchy,
\eqref{eq:truncationcontr}.

\iffalse
\two
We also remark that the $\hbmg_-$ extension is formal due to 
a certain product of the form $\C((\lambda^{-1})) \cdot \C((\lambda))$
which appear in the extended structure equation for $\bmcs$, \eqref{eq:dS+}.
The $\hbmg_+$ extension on the other hand is geometric and 
admits the legitimate local expressions. 
\fi

%--------------------------------------------------------
%--------------------------------------------------------
\sub{Definition}
Set the affine lifts of the Killing fields by
\begin{align}
\hbY&:=(0, \lambda^{-1}\bY), \qquad \tn{($\lambda \hbmg^{\pm}_{N+1}$-valued)}\\
\hbmcs_{\pm}&:=(\tne^{u_{\pm}}, \tne^{u_{\pm}}\bmcs),\n  \qquad \tn{($\hbmg^{\pm}_{N+1}$-valued)}
\end{align}
where the conformal factors $u_{\pm}$ are to be determined.
Set the $\hbmg^{\pm}_{N+1}$-valued extended Maurer-Cartan forms by\ftmark\fttext{Here  it is understood that $\sigmab_{\ell}, \sigma_k=0,\;\; \forall \ell, k \leq N.$}%%%%%%%%%%
\begin{align}
&\hspace{5.4cm}\bPhi_{\pm}:=(\bsigma_{\pm}, \;\;\hbphi_{\pm}),&&\\
&\qquad \bsigma_+:=\sum_{\ell=0}^{\infty} -\lambda^{2\ell}\sigmab_{\ell}, 
&&\bsigma_-:=\sum_{k=0}^{\infty} \lambda^{-2k}\sigma_k, \n\qquad \\
&\qquad \hbphi_+:= \sum_{\ell=0}^{\infty}-\ol{\mcs}^t_{\ell}\sigmab_{\ell}
+\bphi,  
&&\hbphi_-:=\bphi+\sum_{k=0}^{\infty}\mcs_k\sigma_k.\n\qquad
\end{align}
\begin{defn}\label{defn:eCMC}
The  \tb{extended CMC hierarchy} is the system of differential equations,
\begin{align} 
\ed \hbY +[\bPhi_{\pm},  \hbY]&=0, \label{eq:Yeq}    \\
\ed \hbmcs_{\pm} +[\bPhi_{\pm}, \hbmcs_{\pm} ]&=0, \label{eq:Seq}  \\
\ed\bPhi_{\pm}+\frac{1}{2}[\bPhi_{\pm},\bPhi_{\pm}]&=0. \label{eq:Phieq} 
\end{align}
It states that $\hbY, \hbmcs_{\pm}$ satisfy the Killing field equation
for the $\hbmg^{\pm}_{N+1}$-valued extended Maurer-Cartan forms $\bPhi_{\pm},$
and that $\bPhi_{\pm}$ satisfy the compatibility equation.
\end{defn}
We mention that, modulo the additional $\sigmab_{\ell}, \sigma_k$-terms, 
these equations agree with the affine lifts
described in \S\ref{sec:affinelift}.

\two
We are now ready to state the main theorem of this paper.
\begin{thm}\label{thm:main}
The system of equations \eqref{eq:Yeq}, \eqref{eq:Seq}, \eqref{eq:Phieq} 
for the extended CMC hierarchy  is compatible, i.e., 
$\ed^2=0$ is a formal consequence of the structure equation.
\end{thm}

Before we proceed to the proof, let us introduce a convention for  simplified notations.
%--------------------------------------------------------
%--------------------------------------------------------
\sub{Notation}
For a uniform treatment, we introduce  the dummy notations 
$( u, \hbmcs,  \bsigma,  \hbphi, \bPhi)$  without   $\pm$-sign as follows.

\two
Set
\be
\hbmcs:=(\tne^{u }, \tne^{u}\bmcs),\n  %\qquad \tn{($\hbmg^{\pm}_{N+1}$-valued)}
\ee
without $\pm$-sign, where the conformal factor $u$ is to be determined.
Set 
\begin{align}\label{eq:Phinotation}
\bPhi&:=(\bsigma, \;\;\hbphi),&&\\
 \bsigma&:=\sum_{\ell=0}^{\infty} -\lambda^{2\ell}\sigmab_{\ell} 
+\sum_{k=0}^{\infty} \lambda^{-2k}\sigma_k, \n\qquad \\
\hbphi&:= \sum_{\ell=0}^{\infty}-\ol{\mcs}^t_{\ell}\sigmab_{\ell}
+\bphi+\sum_{k=0}^{\infty}\mcs_k\sigma_k.\n\qquad
\end{align}
It is understood that either
$$\begin{cases}
&\tn{$\sigma_k=0, \,\forall \,k,$ and all these objects are  with $(+)$-sign}, \\ 
&\quad\tn{or} \\
&\tn{$\sigmab_{\ell}=0, \,\forall \,\ell,$ and all these objects are   with $(-)$-sign}.
\end{cases}$$
We also use the dummy notation  for
the truncated Virasoro algebra $\Vir_{N+1}$,  
and the affine Kac-Moody algebra $\hbmg_{N+1}$, etc.

%--------------------------------------------------------
%--------------------------------------------------------
\sub{Extended structure equation for $\bY$}
We check that Eq.\eqref{eq:Yeq} for $\hbY$ agrees with 
Eq.\eqref{eq:Y+}.

\two
By definition of the Lie bracket for $\hbmg_{N+1}$, Eq.\eqref{eq:Yeq} becomes
$$\ed (\lambda^{-1}\bY)+[\hbphi,\lambda^{-1}\bY]
+\bdot{(\lambda^{-1}\bY)}\bsigma=0.$$
Multiplying by $\lambda$, we get
\be\label{eq:dY+}
\ed \bY +[\hbphi,\bY]
+(\bdot{ \bY}-\bY)\bsigma=0.\ee
This is equivalent to Eq.\eqref{eq:Y+}.
%--------------------------------------------------------
%--------------------------------------------------------
\sub{Extended structure equation for  $(u, \pmb{\mcs})$}\label{sec:eS}
We expand Eq.\eqref{eq:Seq} to the structure equations for $(u,\bmcs)$.
%--------------------------------------------------------%\marg{  eq:u, eq:S+}
\subb{Conformal factor $u$} 
The derivation part of \eqref{eq:Seq} gives
$$\ed(\tne^u)\tcr{-}\tne^u\bdot{\bsigma}\tcr{+}\bdot{(\tne^u)}\bsigma=0.
$$
Multiplying  by $\tne^{-u}$, we get
\be\label{eq:u}
\ed u =\tcr{+}\bdot{\bsigma}\tcr{-}\bdot{u}\bsigma.
\ee
Recall the structure equation for  $\Vir_{N+1}$,
\be\label{eq:sigmastrt2}
\ed\bsigma+\bsigma\w\bdot{\bsigma}=0.
\ee
Eq.\eqref{eq:u} is compatible  with this structure equation.

\two
%\marg{when $u$ is $\C((\lambda))$-valued,
%for \eqref{eq:Seq}}%%%%%%%%%%%%%%%

%--------------------------------------------------------
\subb{Spectral Killing field $\bmcs$}
The loop algebra part of  \eqref{eq:Seq} gives
$$\ed(\tne^u\bmcs)+[\hbphi,\tne^u\bmcs]+\bdot{(\tne^u\bmcs)}\bsigma-\tne^u\bdot{\hbphi}=0.$$
Multiplying by $\tne^{-u}$, we get
\begin{align} 
0&=\ed\bmcs+[\hbphi,\bmcs]+\bmcs\ed u +(\bdot{\bmcs}+\bdot{u}\bmcs)\bsigma-\bdot{\hbphi}\n\\
%&=\ed\bmcs+[\hbphi,\bmcs]+\bmcs\ed u +(\bdot{\bmcs}+\bdot{u}\bmcs)\bsigma-\bdot{\hbphi}\n\\
&=\ed\bmcs+[\hbphi,\bmcs]+\bmcs(\tcr{+}\bdot{\bsigma}\tcr{-}\bdot{u}\bsigma) +(\bdot{\bmcs}+\bdot{u}\bmcs)\bsigma-\bdot{\hbphi}.\n
\end{align}
The extended structure equation for $\bmcs$ becomes
\be\label{eq:dS+}
\ed\bmcs+[\hbphi,\bmcs]+(\tcr{+}\bmcs\bdot{\bsigma}  +\bdot{\bmcs}\bsigma)=\bdot{\hbphi}.
\ee
%--------------------------------------------------------
\sub{Structure equation for $\hbphi$}
The loop algebra part of \eqref{eq:Phieq}  gives the following structure equation for $\hbphi$:
\be\label{eq:hbphi+0}
\ed\hbphi +\hbphi\w\hbphi=\bdot{\hbphi}\w\bsigma.
\ee

\two
A direct computation shows that
$$\ed(\tn{Eq.}\eqref{eq:dY+}), \ed(\tn{Eq.}\eqref{eq:dS+}) 
\equiv 0\mod\; \eqref{eq:sigmastrt2}, \eqref{eq:hbphi+0}.$$
In turn, it will be shown that Eq.\eqref{eq:hbphi+0} is compatible\ftmark
\fttext{In Eq.\eqref{eq:hbphi+0}, the terms of $\lambda$-degree $\ne 0$ are identity modulo 
Eqs.\eqref{eq:dY+}, \eqref{eq:sigmastrt2}, \eqref{eq:dS+}. 
The remaining terms of $\lambda$-degree $0$ determine  the structure equation 
for $\ed\rho$, which is compatible with  Eqs.\eqref{eq:dY+}, \eqref{eq:sigmastrt2}, \eqref{eq:dS+}. }
with Eqs.\eqref{eq:dY+}, \eqref{eq:sigmastrt2}, \eqref{eq:dS+}.

%--------------------------------------------------------
%--------------------------------------------------------
\section{Additional affine Killing fields}\label{sec:addaffine}
In addition to the affine extension for $\bY, \bmcs$,
we  record in this section 
the affine extension for the additional Killing fields $\bP_{\pm}$,
the dressed Killing fields $\bV_{\pm}$, and the dressed normalized spectral Killing field $\cbmcs$. 
As a result, it will be shown that
the entire dressing process  from $(\bY, \bV_{\pm}, \cbmcs)$ to
$(\bY, \bP_{\pm}, \bmcs)$, and the algebraic formulas \eqref{eq:Rcformula}, \eqref{eq:RVformula}  for $\cbmcs, \bmcs$ admit the compatible affine extension,
while preserving the Lie bracket relations among $\{\bY, \bV_{\pm}, \bP_{\pm}, \cbmcs, \bmcs\}$.

%--------------------------------------------------------
\sub{$\det(\bY)$}
We claim that the determinant constraint
\be\label{eq:detY}
\det(\bY)=-4\gamma\lambda^2
\ee
is compatible with the extended structure equation \eqref{eq:dY+}.
We can therefore continue to impose \eqref{eq:detY} for the extended CMC hierarchy.

\two
In order to verify this, note from \eqref{eq:dY+} the identity
$$\ed\left(\lambda^{-1}\bY\right)^2=-\bdot{\left(\lambda^{-1}\bY\right)^2 }\bsigma.
$$
Here we used the fact that $\bY^2+\det(\bY)I_2=0$, for $\tr(\bY)=0$.
Form this, one finds
$$\ed\left(\det(\lambda^{-1}\bY) \right)=-\bdot{ \left(\det(\lambda^{-1}\bY) \right) }\bsigma.
$$
This is compatible with \eqref{eq:detY}.

\two
From now on, we continue to assume the constraint \eqref{eq:detY} for the extended CMC hierarchy.
This in particular allows one to   use the adjoint operator $\tn{ad}_{\bY}$
and its eigen-matrices as before  in the analysis of the (dressed) 
additional Killing fields for  the extended CMC hierarchy.

%--------------------------------------------------------
\sub{Extended structure equation for $\bP_{\pm}$}
The Lie bracket relations \eqref{eq:YPP} suggest the following
as the affine lifts of the additional Killing fields $\bP_{\pm}$:
\begin{align}\label{eq:hPpm}
\hbP_+&:=(0, \bP_+),  \\
\hbP_- &:=(0, \lambda^{-1}\bP_-).\n
\end{align}
The corresponding affine Killing field equations reduce to,
\begin{align}\label{eq:dP+2}
\ed \bP_+ +[\hbphi,\bP_+]+ \bdot{\bP}_+ \bsigma=0, \\
\ed \bP_- +[\hbphi,\bP_-]+(\bdot{ \bP}_--\bP_-)\bsigma=0.\n
\end{align}

It is easily checked that the extended structure equations \eqref{eq:dY+},  \eqref{eq:dP+2}
are compatible with the Lie bracket relations \eqref{eq:YPP} 
and the determinant formulas \eqref{eq:PPdet}.
Moreover, combined with the extended structure equation \eqref{eq:dS+} for $\bmcs$,
they are also compatible with the Lie bracket relations \eqref{eq:mus}.

\two
From now on, we continue to assume the algebraic relations \eqref{eq:YPP}, \eqref{eq:PPdet}, \eqref{eq:mus}  for the extended CMC hierarchy.

%--------------------------------------------------------
\sub{Extended dressing}
Next, we derive the extended structure equations for the dressed Killing fields
$\bV_{\pm}$, \eqref{eq:Vpm}. Then, it will be shown that
this affine extension is also compatible with the algebraic formulas \eqref{eq:Rcformula}, \eqref{eq:RVformula}  for  $\cbmcs, \bmcs$.

\two
Define $\cbphi$ by the equation
\be\label{eq:hbphi}
\hbphi-\cbphi:=\bY\alpha+\bmcs\bsigma.
\ee
This extends the identity \eqref{eq:Yidentity}, and $\cbphi$ defined here
can be considered as an affine extension of $\cbphi$ defined earlier for dressing.

\begin{lem}\label{lem:cbphi}$\,$
\benu[\qquad a)]
\item By definition,
\be\label{eq:cbphiY}
\ed\bY+[\cbphi,\bY]=0.\ee
\item
The $\bmg_{\geq 0}$-valued 1-form $\cbphi$ satisfies the structure equation
$$\ed\cbphi+\cbphi\w\cbphi=0.$$
\enu
\end{lem}
\begin{proof} 
a) Eq.\eqref{eq:dY+} can be written as,
\be%\label{eq:cbphiY}
\ed\bY+[\hbphi-\bmcs\bsigma,\bY]=\ed\bY+[\cbphi+\bY\alpha,\bY]=0.\n\ee

b) Using \eqref{eq:cbphiY},
\begin{align*}
 \ed\hbphi+\hbphi\w\hbphi-\bdot{\hbphi}\w\bsigma&=  \ed\cbphi
-[\cbphi,\bY]\w\alpha-([\cbphi+\bY\alpha,\bmcs]+\bmcs\bdot{\bsigma}-\bdot{\hbphi})\w\bsigma
-\bmcs\bsigma\w\bdot{\bsigma}\\
&\quad+\cbphi\w\cbphi+[\cbphi,\bY]\w\alpha+[\cbphi,\bmcs]\w\bsigma 
+[\bY,\bmcs]\alpha\w\bsigma-\bdot{\hbphi}\w\bsigma \\
&= \ed\cbphi+\cbphi\w\cbphi.
\end{align*}
\end{proof}

Recall from \eqref{eq:bsigma+normal} the normalization
$$ 4 e^{2\bsigma^+}\bbb\bbc=1.$$
Let
\be\label{eq:cbbp}
\ed\bbp:=
\sqrt{\gamma}\left(\frac{\bbb\cbphi^1_2+\bbc\cbphi^2_1}{\bbb\bbc}\right)
\ee
be the affine extension of the non-local variable $\bbp$ defined in \eqref{eq:bbp}
(which is defined by the same formula).
\iffalse
Expanding  as a series in $\lambda$, one finds
$$\ed\bbp=  \left( \frac { \left(  h_{{2}} b_{{4}}
-\gamma  c_{{4}} \right) \xi }{2\sqrt {\gamma}  }
+ \frac { \im \left(  h_{{2}}\hb_2+ {\gamma}^{2} \right) \xib
}{2\sqrt {\gamma}h_2^{\frac{1}{2}} }\right)+\mco(\lambda^2)\mod \ed \tb{t},\ed\ol{\tb{t}}.
$$
This suggests that $\ed\bbp$ contains all the higher-order conservation laws,  
and hence that $\bbp$ is indeed a non-local function.

Set
\begin{align}\label{eq:cVpm}
\bV_+&=e^{\bsigma^+}
\left[ \begin {array}{cc} 2 \sinh \left( \lambda \bbp \right) \lambda^{-1}\bbb\bbc
& 2 \sqrt {\gamma}\cosh \left( \lambda \bbp \right)\bbc 
+\im \sinh \left( \lambda \bbp \right) \lambda^{-1}\bba\bbc
\\\noalign{\medskip}
 -2 \sqrt {\gamma}\cosh \left( \lambda \bbp \right) \bbb
+\im\sinh \left( \lambda \bbp \right)\lambda^{-1}  \bba\bbb
&-2 \sinh \left(  \lambda \bbp \right) \lambda^{-1}\bbb\bbc \end {array}
 \right] , \\
\bV_-&=e^{\bsigma^+}
\left[ \begin {array}{cc} -2 \cosh \left( \lambda \bbp \right)  \bbb\bbc
& -2 \sqrt {\gamma}\sinh \left( \lambda \bbp \right)\lambda \bbc 
-\im \cosh \left(  \lambda \bbp \right)  \bba\bbc
\\\noalign{\medskip}
 2 \sqrt {\gamma}\sinh \left( \lambda \bbp \right) \lambda\bbb
-\im\cosh \left( \lambda \bbp \right)  \bba\bbb
&2 \cosh \left( \lambda \bbp \right)  \bbb\bbc \end {array}
 \right]. \n
\end{align}
\fi
Define the corresponding  extension of $\bV_{\pm}$ by the same formula \eqref{eq:Vpm}.
Then, we have the following extension of Thm.\ref{thm:VV}.
\begin{thm}\label{thm:cVV}
Let $\bV_{\pm}$ be  defined by \eqref{eq:cbbp}, \eqref{eq:Vpm}. 
They satisfy the Killing field equation for $\cbphi$,
$$\ed\bV_{\pm}+[\cbphi,\bV_{\pm}]=0.$$
The set of three Killing fields $\{\, \bY, \bV_{\pm}\}$ generates the space of 
$\bmg_{\geq 0}$-valued Killing fields for  $\cbphi.$
%Note that $\bV_+$ is $\bmg_{\geq 0}$-valued, and $\bY, \bV_-$ are $\bmg_{\geq 1}$ valued.
\end{thm}
Similarly as for $\{\bY, \bP_{\pm}\}$, 
we continue to assume the algebraic  relations
\eqref{eq:YVV}, \eqref{eq:YVYV'}, \eqref{eq:VVdet}.

\two
Let $\cbmcs$ be the corresponding extension of the dressed spectral Killing field 
defined by the same formula \eqref{eq:Rcformula}. Then,
\begin{cor}\label{cor:dcS+}
The affine extension  $\cbmcs$ satisfies the spectral Killing field equation for $\cbphi$,
\be\label{eq:dcS+}
\ed\cbmcs+[\cbphi,\cbmcs]=\bdot{\cbphi}.
\ee
\end{cor}
 
%----------------------------------------------------------------------------
%----------------------------------------------------------------------------
\sub{Additional Killing fields for $\pmb{\Phi}$ }%$\tn{$\bphi$}$}
Summarizing the analysis so far,
the additional (formal) solutions to the Killing field equation for $\bPhi$
can be obtained from $\{  \bV_{\pm}\}$.

\two
%Recall
%$$\ttt=\frac{1}{2\im}\sum_{m=0}^{\infty}\lambda^{-(2m+2)}t_m =\int\balpha.$$
Let
\be\label{eq:cPPpm}
\bp \bP_+ \\ \bP_-\ep:=
\bp  
\cosh(4\sqrt {\gamma}\lambda\ttt)  &  -  \sinh(4\sqrt {\gamma}\lambda\ttt)\lambda^{-1}     \\
 -\sinh(4\sqrt {\gamma}\lambda\ttt)\lambda& \cosh(4\sqrt {\gamma}\lambda\ttt)       \\
\ep
\bp \bV_+ \\ \bV_-\ep 
\ee
be the affine extension of the Killing fields $\bP_{\pm}$ defined 
by the same formula as in \eqref{eq:PPpm}.
%(which we denote by the same notation).
They also satisfy the algebraic relations \eqref{eq:YPP}, \eqref{eq:YPYP'}, \eqref{eq:PPdet}.
 
\two
Recall the affine lifts $\hbP_{\pm}$, \eqref{eq:hPpm}.
Then we have the following extension of Thm.\ref{thm:PP}.
\begin{thm}\label{thm:hPP}
Let $\hbP_{\pm}$ be  defined  by \eqref{eq:hPpm}, \eqref{eq:cPPpm}, \eqref{eq:cbbp}, \eqref{eq:Vpm}. 
They satisfy the Killing field equation for the affine Maurer-Cartan form $\bPhi$,
$$\ed\hbP_{\pm}+[\bPhi,\hbP_{\pm}]=0.$$
Modulo $\hbmcs$,
the set of three Killing fields $\{\,\hat{\bY},  \hbP_{\pm}\}$ generates
the space of $\Vir_{N+1}\ltimes\sla(2,\C)[[\lambda^{-1},\lambda]]$-valued 
Killing fields for  $\bPhi$ 
whose loop algebra parts are of the form \eqref{eq:oftheform}.\ftmark\fttext{In this case, 
$\bQ_0, \bQ_{\pm}$ are $\sla(2,\C)[[\lambda^{-1},\lambda]]$-valued.}
\end{thm}
\begin{proof}
The structure equations for $\bY$, \eqref{eq:dY+}, and $\bP_{\pm}$, \eqref{eq:dP+2}, 
are equivalent to
\begin{align}\label{eq:dP+2'}
\ed \bY +[\hbphi-\bmcs\bsigma,\bY]&=0,  \\
\ed \bP_{\pm} +[\hbphi-\bmcs\bsigma,\bP_{\pm}]&=0. \n
\end{align}
Substitute \eqref{eq:cPPpm}, and Eqs.\eqref{eq:dP+2'} reduce to the identity
\be\label{eq:hcidentity}
 \hbphi-\bmcs\bsigma=\cbphi+\bY\alpha.
\ee
\end{proof}
Note also that the structure equation for $\bmcs$, \eqref{eq:dS+}, can be written as
\be\label{eq:dS+'}
\ed\bmcs+[\hbphi-\bmcs\bsigma, \bmcs]=\bdot{\left(\hbphi-\bmcs\bsigma\right)}.
\ee
Comparing this with \eqref{eq:dcS+}, one gets  the following extension of  Thm.\ref{thm:spectralKilling}.
\begin{cor}\label{thm:spectralKilling+}
The affine extension of the normalized spectral Killing field $\bmcs$ is given
by the same formula \eqref{eq:RVformula}.
\end{cor}

\one
This completes the affine extension of the entire dressing process.

%--------------------------------------------------------
\sub{Non-Abelian integrable extension for $\bbp$}
We wish to give an interpretation of the defining equation  for 
the non-local variable $\bbp$. 
For the general reference on integrable extension,
 \cite{Vinogradov1989}  
  
\two
Recall the formula for $\ed\bbp$, 
\be\label{eq:dpagain}
\ed\bbp =
\sqrt{\gamma}\left(\frac{\bbb\cbphi^1_2+\bbc\cbphi^2_1}{\bbb\bbc}\right).
\ee
Consider first the original (un-extended) CMC hierarchy case.
From the equation 
$$\hbphi-\cbphi\equiv \bY\alpha \mod \bsigma, $$
in this case the RHS of \eqref{eq:dpagain} consists of the terms in 
$\{\bba, \bbb, \bbc, \ed t_m \}$ and their complex conjugates only. 
It follows that the RHS represents an infinite sequence of local conservation laws for the CMC hierarchy, and $\bbp$ is the potential for the corresponding Abelian integrable extension.

For the extended CMC hierarchy case on the other hand,
the RHS contains the terms involving   $\bmcs$.
The formulas \eqref{eq:RVformula}, \eqref{eq:Rcformula}, \eqref{eq:Vpm} then imply that
the RHS involves $\{\bbp, \bdot{\bbp}\}$.
Thus, the integrable extension for $\bbp$ defined by \eqref{eq:dpagain} 
is a non-Abelian extension.

%--------------------------------------------------------
%--------------------------------------------------------
\section{Extended CMC hierarchy}\label{sec:exformulas}
For convenience,
we collect here the whole structure equations for the extended CMC hierarchy.

\one\noi
[\tb{Structure equations for $\bY, \bmcs, \hbphi$}]
%\marg{include Lie bracket relations, and formulas for $\bP_{\pm}$, $\bmcs$, etc}%%%%%%%%%
\begin{align} 
\label{eq:dY+2}
&\ed \bY +[\hbphi,\bY]
+(\bdot{ \bY}-\bY)\bsigma=0. \\
\label{eq:dS+2}
&\ed\bmcs+[\hbphi,\bmcs]+(\tcr{+}\bmcs\bdot{\bsigma}  +\bdot{\bmcs}\bsigma)=\bdot{\hbphi}. \\
\label{eq:u2}
&\ed\bsigma+\bsigma\w\bdot{\bsigma}=0,  
\quad\ed u =\tcr{+}\bdot{\bsigma}\tcr{-}\bdot{u}\bsigma.\\
\label{eq:hbphi+}
&\ed\hbphi +\hbphi\w\hbphi=
\bdot{\hbphi}\w\bsigma.
\end{align}
%Eq.\eqref{eq:hbphi+} is the loop algebra part of  \eqref{eq:Phieq}.
%One finds that  $\ed^2=0$ is an identity for Eqs.\eqref{eq:dY+2}, \eqref{eq:dS+2}
%modulo Eqs.\eqref{eq:u2}, \eqref{eq:hbphi+}.

\one\noi
[\tb{Structure equations for $\bP_{\pm}, \bV_{\pm}, \cbmcs$}]
%\marg{include Lie bracket relations, and formulas for $\bP_{\pm}$, $\bmcs$, etc}%%%%%%%%%
\begin{align} 
\label{eq:hcidentity+}
&\hbphi-\bmcs\bsigma=\cbphi+\bY\alpha. \\
\label{eq:dP+2''}
&\ed \bP_{\pm} +[\hbphi-\bmcs\bsigma,\bP_{\pm}]=0.\\
\label{eq:dS+''}
&\ed\bmcs+[\hbphi-\bmcs\bsigma, \bmcs]=\bdot{\left(\hbphi-\bmcs\bsigma\right)}.\\
\label{eq:dV+2''}
&\ed \bV_{\pm} +[\cbphi,\bV_{\pm}]=0.\\
\label{eq:dcS+''}
&\ed\cbmcs+[\cbphi, \cbmcs]=\bdot{ \cbphi}.\\
%\quad\ed u =\tcr{+}\bdot{\bsigma}\tcr{-}\bdot{u}\bsigma,\\
\label{eq:cbphi''}
&\ed\cbphi +\cbphi\w\cbphi=0.
\end{align}

\two\noi
[\tb{Algebraic relations}]
\begin{align} 
\label{eq:YPP''}
&[\bY,\bP_+]= 4\sqrt{\gamma}  \bP_-, \quad
[\bY,\bP_-]= 4\sqrt{\gamma}\lambda^2 \bP_+, 
\quad[\bP_+,\bP_-]=\sqrt{\gamma} \bY,\\
%\label{eq:YPYP''}
&\bY \bP_+= -\bP_+\bY = 2\sqrt{\gamma}  \bP_-, \quad
\bY \bP_- =-\bP_-\bY= 2\sqrt{\gamma}\lambda^2 \bP_+,  
 \quad \bP_+ \bP_- =-  \bP_- \bP_+=\frac{1}{2}\sqrt{\gamma} \bY.\n\\
\label{eq:PPdet''}
&\det(\bP_+)= \gamma, \qquad 
\det(\bP_-)= -\gamma \lambda^2, \qquad \det(\bY)=\bba^2-4\bbb\bbc=-4\gamma\lambda^2 \\
&\tn{(the same relations for $\{\bV_{\pm}, \bY \}$)}.\n \\
\label{eq:Rformula''}
&\bmcs=c_0\ad_{\bY}(\bdot{\bY})+c_+\ad_{\bP_+}(\bdot{\bP}_+)+c_-\ad_{\bP_-}(\bdot{\bP}_-), 
\qquad \tn{(the similar formula for $\cbmcs$)}\\
%\cbmcs=c_0\ad_{\bY}(\bdot{\bY})+c_+\ad_{\bV_+}(\bdot{\bV}_+)+c_-\ad_{\bV_-}(\bdot{\bV}_-),\\
&
c_0=+\frac{1}{32\gamma}\lambda^{-2},  \quad
c_{+}=-\frac{1}{8\gamma},
\quad
c_{-}=\frac{1}{8\gamma}\lambda^{-2}. \n\\
\label{eq:tdt''}
&\ttt=\frac{1}{2\im}\sum_{m=0}^{\infty} \lambda^{-(2m+2)}t_m,\qquad
\bdot{\ttt}+\ttt=\frac{\im}{2}\sum_{m=0}^{\infty}(2m+1)\lambda^{-(2m+2)}t_m. \\
\label{eq:mus''}
&\bmcs=\left(\bdot{\ttt}+\ttt \right)\bY+\check{\bmcs}.\\
\label{eq:PS''}
&
[\bP_+, {\bmcs}]=\bdot{\bP}_+,\quad [\bP_-, {\bmcs}]=\bdot{\bP}_- -\bP_-,
\quad[\bY, {\bmcs}]=\bdot{\bY}-\bY  \\
&\tn{(the same relations for $\{\bV_{\pm}, \cbmcs\}$)}.\n \\
\label{eq:key''}
&[ {\bmcs},\lambda^{-1}\bY]+ \bdot{(\lambda^{-1} \bY)}=0.
\end{align}

%--------------------------------------------------------
\sub{Formal sum}\label{sec:formalvs}
Consider the original un-truncated CMC hierarchy, i.e.,  we set $u=0, \bsigma=0$,
and let the truncation parameter $N\to \infty.$
Let $\mct_{\infty}=\{\del_{t_m}\}_{m=0}^{\infty}, \ol{\mct}_{\infty}=\{\del_{\tba_n}\}_{n=0}^{\infty}$ be the (commutative) Lie algebra of vector fields 
formally dual to the 1-forms
$\{\ed t_m\}_{m=0}^{\infty}, \{\ed \tba_n\}_{n=0}^{\infty}$ respectively.

From this point of view, the truncation process described in \S\ref{sec:truncated}
amounts to restricting the CMC hierarchy to a submanifold
which is tangent to the subalgebras of vector fields  
$\mct_{N}=\{\del_{t_m}\}_{m=0}^{N}, \ol{\mct}_{N}=\{\del_{\tba_n}\}_{n=0}^{N}$ respectively.
This essentially relies on the fact that the infinite dimensional Lie algebras 
$\mct_{\infty}, \ol{\mct}_{\infty}$
support such finite dimensional subalgebras.

On the other hand,  from the structure equation  \eqref{eq:ssbracket} or \eqref{eq:sigmastrt},
the Virasoro algebras $\Vir^{\pm}_{N+1}$ do  not appear to support any obvious family of 
finite dimensional subalgebras (of dimension $\geq 2$).
The finite truncation process in \S\ref{sec:truncated} is not as compatible 
for the non-commuting Virasoro algebras.

\two 
Under this circumstance,
we wish to check if the structure  equations presented above have any 
convergence related issues.
For this purpose, we  examine the dressed Maurer-Cartan form $\cbphi$,
for the variable $\bbp$ is defined by \eqref{eq:dpagain}
and it is used in the definition of $\cbmcs, \bmcs$.
 
By definition \eqref{eq:hcidentity+},
$$\cbphi\equiv \hbphi-\bmcs\bsigma \mod\; \ed t_m, \ed \tba_n,\;\forall\; m,n.$$
Substitute \eqref{eq:Phinotation}, and we get
\begin{align}
\cbphi&\equiv \sum_{\ell=0}^{\infty}-\ol{\mcs}^t_{\ell}\sigmab_{\ell}
 +\sum_{k=0}^{\infty}\mcs_k\sigma_k 
-\bmcs(\sum_{\ell=0}^{\infty} -\lambda^{2\ell}\sigmab_{\ell}  
+\sum_{k=0}^{\infty} \lambda^{-2k}\sigma_k) \mod\; \ed t_m, \ed \tba_n,\;\forall\; m,n,\n\\
&=\sum_{\ell=0}^{\infty}(-\ol{\mcs}^t_{\ell}+\lambda^{2\ell}\bmcs )\sigmab_{\ell}
+ \sum_{k=0}^{\infty}-\mcs_{(k+1)}\sigma_k.\n
\end{align}

Consider the $\sigma_k$-term.
In this case,  $\mcs_{(k+1)}$ is $\bmg_{\geq 0}$-valued. Hence
for each $\lambda$-degree $\geq 0$, 
the expression $\sum_{k=0}^{\infty}-\mcs_{(k+1)}\sigma_k$ contains 
an infinite sequence of $\sigma_k$-terms of the given degree.

Consider next the $\ol{\sigma}_{\ell}$-term.
In this case,  $\ol{\mcs}^t_{\ell}$ is $\bmg_{\geq 1}$-valued. Hence
for each $\lambda$-degree $\geq 1$, 
the expression $\sum_{\ell=0}^{\infty}(-\ol{\mcs}^t_{\ell}+\lambda^{2\ell}\bmcs )\sigmab_{\ell}$ 
also contains an infinite sequence of $\sigmab_{\ell}$-terms of the given degree.

\two
From this, we conclude that 
the coefficients of the $\C[[\lambda^{-2},\lambda^2]]$-valued 
1-form $\ed\bbp$ are defined generally  as a formal sum.
This implies that \emph{the affine extension of the truncated CMC hierarchy
presented above should generally be considered as a formal system of equations.}
\iffalse
\begin{rem}
Consider for an exceptional example the case
when we choose one $\sigmab_{\ell}$ and set $\sigmab_s=0\;\forall s\ne\ell$, i.e.,
restrict to a one dimensional subalgebra of the Virasoro algebra
\end{rem}
\fi

%\np
%%%%%%%%%%%%%
%%%%%%%%%%%%% 
%%%%%%%%%%%%% 
%%%%%%%%%%%%% 
%%%%%%%%%%%%%%%%%%%%%%%%%% 
\section{Proof of compatibility}\label{sec:proof}
From \eqref{eq:hbphi+}, set
\be\label{eq:LRHS}
\tn{LHS}:=\ed\hbphi +\hbphi\w\hbphi=
\bdot{\hbphi}\w\bsigma =:\tn{RHS}.
\ee
The claim is that  (for the terms of $\lambda$-degree $\ne 0$),
$$\tn{LHS}\equiv \tn{RHS} \mod \eqref{eq:dY+2}, \eqref{eq:dS+2}, \eqref{eq:u2}.
$$
We check only the  
$\ed t_m\w\sigma_k, \ed t_m\w\sigmab_{\ell},\sigma_j\w\sigma_k$-terms.
Compatibility of the remaining $ \ed \tba_n\w\sigma_k, \ol{\sigma}_{\ell}\w\ol{\sigma}_s$-terms
follow  by taking the formal conjugate transpose. 

The terms of $\lambda$-degree 0 in \eqref{eq:LRHS} determine  the formula for $\ed\rho$,
the exterior derivative of the connection form $\rho$.
This is treated in \S\ref{sec:drho}.

%--------------------------------------------------------
%--------------------------------------------------------
\sub{$\ed t_m\w\sigma_k$-terms}
These terms are checked in \S\ref{sec:tmskY}. Note that this part of the compatibility
relies on the condition
\be\label{eq:kconstraint}
 k \geq 0.\ee
%--------------------------------------------------------
%--------------------------------------------------------
\sub{$\ed t_m\w\sigmab_{\ell}$-terms}
The terms of $\lambda$-degree 0 contribute only to $\ed \rho$, which will be checked later.
We only consider the terms  of $\lambda$-degree $\ne 0.$
We wish to show that this part of the compatibility relies on the condition
\be\label{eq:lconstraint}\ell \geq N+1.
\ee

\two
The RHS gives
$$\tn{RHS}= -\lambda^{2\ell}\bdot{U}_m.$$
The LHS gives 
$$-\del_{t_m}\ol{S}_{\ell}^t-\delsb{\ell} U_m+[\ol{S}_{\ell}^t, U_m].$$

Note the identity
$$\ol{(\lambda^{2\ell}\bdot{U}_m)_{\geq 1}}^t=-(\lambda^{-2\ell} \bdot{\ol{U}^t}_m)_{\leq -1},$$
i.e., the upper-dot and the formal complex conjugation operations anti-commute.
Applying this to the conjugate transpose of \eqref{eq:delSk}, one gets
$$-\del_{t_m}\ol{S}_{\ell}^t=[U_m, \ol{S}_{\ell}^t]_{\geq 1}-(\lambda^{2\ell}\bdot{U}_m)_{\geq 1}.$$
From \eqref{eq:delUm}, we also have
$$-\delsb{\ell} U_m=-[\ol{S}_{\ell}^t,U_m]_{\leq -1}-(\lambda^{2\ell}\bdot{U_m})_{\leq -1}
-(2m+1)(\lambda^{2\ell}U_m)_{\leq -1}.
$$
Collect the terms  of $\lambda$-degree $\ne 0.$
After cancellation, the desired compatibility reduces to,
\be\label{eq:truncationcontr}
(\lambda^{2\ell} U_m)_{\leq -1}=0.
\ee
Since the CMC hierarchy is $t_N, \tba_N$-truncated,
the $\lambda$-degree of $U_m$ is bounded below by $-(2N+1)$.
This agrees with the constraint \eqref{eq:lconstraint}.

%--------------------------------------------------------
%--------------------------------------------------------
\sub{$\sigma_j\w\sigma_k$-terms}
We first record a relevant identity for the $\sigma_j$-derivative of $\mcs_k$.
\begin{lem}
For all $j, k \geq 0$,
\be\label{eq:Smixed}
\dels{j} \mcs_k-\dels{k} \mcs_j + [\mcs_j,\mcs_k ]
+(2k-2j) \mcs_{j+k}=\lambda^{-2k}\bdot{\mcs}_j-\lambda^{-2j}\bdot{\mcs}_k.
\ee
\end{lem}
\begin{proof}
From \eqref{eq:dS+2}, substitute $\bmcs=\lambda^{2k}(\mcs_k+\mcs_{(k+1)})$
for the first two $\bmcs$'s, and substitute  $\bmcs=\lambda^{2j+2k}(\mcs_{j+k}+\mcs_{(j+k+1)})$
for the next two $\bmcs$'s.
Collect the $\sigma_j$-terms, and one gets
$$\dels{j}(\mcs_k+\mcs_{(k+1)})+[\mcs_j,\mcs_k+\mcs_{(k+1)}]
+(\bdot{\mcs}_{j+k}+\bdot{\mcs}_{(j+k+1)})
+ 2k  (\mcs_{j+k}+\mcs_{(j+k+1)})=\lambda^{-2k}\bdot{\mcs}_j.
$$
Collect the $\bmg_{\leq -1}$-terms from this, and one gets
$$\dels{j} \mcs_k +[\mcs_j,\mcs_k ]+[\mcs_j, \mcs_{(k+1)}]_{\leq -1}
+\bdot{\mcs}_{j+k}+ 2k \mcs_{j+k}=\lambda^{-2k}\bdot{\mcs}_j.
$$
Interchange $j, k$ and take the difference, and one finds
\be\label{eq:Smixed0}
\dels{j} \mcs_k-\dels{k} \mcs_j +2[\mcs_j,\mcs_k ]+([\mcs_j, \mcs_{(k+1)}]
+[\mcs_{(j+1)},\mcs_k])_{\leq -1}
+(2k-2j) \mcs_{j+k}=\lambda^{-2k}\bdot{\mcs}_j-\lambda^{-2j}\bdot{\mcs}_k.
\ee
From the trivial identity $[\mcs_j+\mcs_{(j+1)},\mcs_k+\mcs_{(k+1)}]=0$, note
$$[\mcs_j,\mcs_k ]+([\mcs_j, \mcs_{(k+1)}]
+[\mcs_{(j+1)},\mcs_k])_{\leq -1}=0.$$
Hence \eqref{eq:Smixed0} reduces to \eqref{eq:Smixed}.
\end{proof}

Now, consider the $\sigma_j\w\sigma_k$-terms in \eqref{eq:LRHS}.
The RHS gives
$$\tn{RHS}=\bdot{\mcs}_j\lambda^{-2k}-\bdot{\mcs}_k\lambda^{-2j}.
$$
The LHS gives
$$
\tn{LHS}=\dels{j} \mcs_k-\dels{k} \mcs_j + [\mcs_j,\mcs_k ]
+2(\ed\sigma_i)_{\sigma_j\w\sigma_k}\mcs_i,
$$
where $(\ed\sigma_i)_{\sigma_j\w\sigma_k}$ means the coefficient of $\sigma_j\w\sigma_k$
in $\ed\sigma_i$.
The claim LHS=RHS follows from \eqref{eq:Smixed} and \eqref{eq:sigmastrt}.

\sub{$\ed\rho$}\label{sec:drho}
The formula for the 2-form $\ed\rho$ is determined by the $\lambda$-degree 0 terms in
\eqref{eq:hbphi+}.
For the compatibility equation $\ed^2\rho=0$,
it suffices to check that
$\ed^2=0$ is an identity for \eqref{eq:hbphi+},
assuming the compatibility of the rest of the equations.
This follows from  \eqref{eq:hbphi+} itself  and \eqref{eq:sigmastrt2}.
We omit the  details.
 
%--------------------------------------------------------
%--------------------------------------------------------
\sub{$\ed h_2$, and $\ed\xi$}
The analysis thus far shows the compatibility 
of the $\tba_N, t_N$-truncated $(-\ol{\tn{AKS}}^t,\tn{AKS})$-hierarchy  
extended by the $\Vir^{\pm}_{N+1}$-symmetry,
under the constraints that 
$$\ed \tba_0=- \frac{1}{2}\hb_2^{\frac{1}{2}}\xib,\quad \ed t_0=-\frac{1}{2}h_2^{\frac{1}{2}}\xi.$$
For the  formulation of the extended CMC hierarchy, 
we wish to separate this to the structure equations for $\ed h_2, \ed\xi$ respectively.

Since the first coefficients $c^2, b^2$ of $\bY$ are multiples of $h_2^{\frac{1}{2}}, h_2^{-\frac{1}{2}}$
respectively, the formula for $\ed h_2$ is included in \eqref{eq:dY+2} (and hence compatible).
Then, $\ed\xi$ is determined from the equations,
\begin{align}
\xi&=-2h_2^{-\frac{1}{2}} \ed t_0, \n\\
\ed\xi&=h_2^{-\frac{1}{2}-1}\ed h_2\w\ed t_0 \n\\
&=-\frac{1}{2}h_2^{-1}\ed h_2\w\xi. \n
\end{align}
The compatibility equation $\ed^2\xi=0$   follows from  the compatibility of $\ed h_2$.

\two
We proceed to extract the formula for $\ed h_2$, or $\ed c^2$, from \eqref{eq:Y+},
by collecting the terms of $\lambda$-degree 1.

It is clear from the equation for $\delsb{\ell}\bY$ in \eqref{eq:Y+} that
$$\delsb{\ell} c^2=0, \;\forall\; \ell,$$
for $\bdot{\bY}-\bY$ has $\lambda$-degree$\,\geq 2$.

From \eqref{eq:pmrelationS'}, we have
$$\dels{k}\bY=[\mcs_{(k+1)}, \bY].
$$
Collect the terms of $\lambda$-degree 1, and one finds
$$\dels{k}(\bY)_1 =[(\mcs_{(k+1)})_0, (\bY)_1].
$$
Here $(\bY)_1$ means the terms of $\lambda$-degree 1 in $\bY$, etc.

Recall from \eqref{eq:tbY},
$$(\bY)_1=\bp \cdot  &2c^2 \\ 2 b^2 & \cdot \ep.$$
Set 
\be\label{eq:tbS}%\tag{eq:tbY}
\bmcs=\bp -\im\tb{a}_{\bmcs}&2\tb{c}_{\bmcs}\\ 2\tb{b}_{\bmcs}&\im\tb{a}_{\bmcs} \ep,  
\ee
where
\be
\tb{a}_{\bmcs}=\sum \lambda^{2n}a^{2n+1}_{\bmcs},\qquad
\tb{b}_{\bmcs}=\sum \lambda^{2n+1}b^{2n+2}_{\bmcs},\qquad
\tb{c}_{\bmcs}=\sum \lambda^{2n+1}c^{2n+2}_{\bmcs}.\n
\ee
By definition $\lambda^{-2k}\bmcs=\mcs_k+\mcs_{(k+1)}$, and we have
$$(\mcs_{(k+1)})_0=\bp -\im a^{2k+1}_{\bmcs}&\cdot \\ \cdot &\im a^{2k+1}_{\bmcs} \ep.
$$
From this, we obtain
\be\label{eq:delsc2}
\dels{k} c^2=-2\im  a^{2k+1}_{\bmcs} c^2.
\ee
Substitute $c^2=\im h_2^{\frac{1}{2}}$, and one gets
\be\label{eq:delsh2}
\dels{k} h_2=-4\im  a^{2k+1}_{\bmcs} h_2.
\ee
\begin{prop}\label{prop:h2xi}
The extended structure equations for $\xi, h_2$ are,
\begin{align}\label{eq:h2xi}
\ed\xi-\im\rho\w\xi     &=\sum_{m=1}^{\infty}  a^{2m+3}\ed t_m\w\xi
+2\im \sum_{k=0}^{\infty}   a^{2k+1}_{\bmcs}\sigma_k\w\xi, \\
\ed h_2+2\im h_2\rho&=h_3\xi -2 \sum_{m=1}^{\infty}  h_2 a^{2m+3}\ed t_m
-4\im  \sum_{k=0}^{\infty}   h_2 a^{2k+1}_{\bmcs}\sigma_k, \n \\
&= -2 \sum_{m=0}^{\infty}  h_2 a^{2m+3}\ed t_m
-4\im  \sum_{k=0}^{\infty}   h_2 a^{2k+1}_{\bmcs}\sigma_k.\n
\end{align}
\end{prop}
\begin{cor}\label{cor:vircon}
The (formal) deformation of the CMC system induced by the Virasoro symmetries
is conformal and preserves Hopf differential.
\end{cor}

%%%%%%%%%%%%
%%%%%%%%%%%%
%%%%%%%%%%%% 
%%%%%%%%%%%% 
%%%%%%%%%%%% 
%%%%%%%%%%%%%%%%%%%%%%%%%%%%%
\section{Central extension}\label{sec:central}
For an application of the affine extension of the CMC hierarchy,
we examine the (missing) central parts of the extended Killing fields $\hbY, \hbmcs.$
The log of tau function is defined as the central component for $\hbmcs$. 
We find a closed formula for  tau function, \eqref{eq:unextau}.
%--------------------------------------------------------
%--------------------------------------------------------
\sub{Central extension for $\lambda^{-1}\tb{Y}$}
Consider the 1-form
\be\label{eq:varphi0bY}
\varphi^0_{\bY}:=\tn{Res}_{\lambda=0} \tr (\lambda^{-1}\bY\bdot{\hbphi}).
\ee
Here   $\tn{Res}_{\lambda=0}$
is the residue operator that takes the terms of $\lambda$-degree 0.
Since $\lambda^{-1}\bY$ takes values in  $\lambda \bmg$, which is orthogonal to $\bmg$,
it is clear that 
$$\varphi^0_{\bY}=0.$$
Thus the central extension for $\lambda^{-1}\bY$ is trivial.

\two
Consider instead the 1-form
\be\label{eq:varphi0bY}
\varphi_{\bY}:=\tr \left(\bY\bdot{(\hbphi-\bmcs\bsigma)}\right).
\ee
Since $\bY$ is also a Killing field for $\cbphi$, and from the relation \eqref{eq:hcidentity}, 
it follows that 
$$\ed \varphi_{\bY}=0.$$

For the original CMC hierarchy, the 1-form $$ \tr ( \bY\bdot{\bphi})$$ represents the infinite sequence
of  local, higher-order conservation laws. 
Thus $\varphi_{\bY}$ represents the affine extension of these conservation laws.

\iffalse
\begin{rem}
Recall the structure equation \eqref{eq:Ybphi} for the original CMC hierarchy.
Consider the $\C[[\lambda^{-2},\lambda^2]]$-valued 1-form
\be\label{eq:varphibY}
\varphi_{\bY}:=  \tr ( \bY\bdot{\bphi}).
\ee
It is easily checked that
$$\ed\varphi_{\bY}=0.$$
When expanded as a series in $\lambda^{-2},\lambda^2$,
each coefficient of $\varphi_{\bY}$ represents a conservation law 
of the CMC hierarchy.
\end{rem}
\fi

%--------------------------------------------------------
%--------------------------------------------------------
\sub{Tau function}
Consider the 1-form
\be\label{eq:varphibS}
\varphi_{\bmcs}:=\tcr{-} \tn{Res}_{\lambda=0} \tr (\tne^u\bmcs \bdot{\hbphi}).
\ee
We claim that,
\begin{lem}
The   1-form $\varphi_{\bmcs}$ is closed,
$$\ed \varphi_{\bmcs}=0.$$
%and it represents a non-local conservation law.
\end{lem}
\begin{proof}
We show that $\ed\varphi_{\bmcs}$ is the residue ($\tn{Res}_{\lambda=0}$)  of a total derivative under 
the derivation $\mcd=\lambda\dd{}{\lambda}$.

Differentiate $\tr( \tne^u\bmcs \bdot{\hbphi})$ and collecting terms, one finds
$$
- \ed \left( \tr (  \tne^u\bmcs \bdot{\hbphi})   \right) =
-\tr \left( \ed ( \tne^u\bmcs \bdot{\hbphi} ) \right) =
\tr\bdot{\left(\tne^u\bmcs\bsigma\w\bdot{\hbphi}\right)}.
$$
\end{proof}
\begin{defn}\label{defn:tau}
The \tb{tau function} $\tau$ for the extended CMC hierarchy is defined by the equation
$$\ed \log(\tau):=\varphi_{\bmcs}.
$$
Thus $\log(\tau)$ is the potential for the closed 1-form $\varphi_{\bmcs}$.
\end{defn}
Here $\varphi_{\bmcs}, \tau$ are also the dummy notations without $\pm$-sign.
 
\two
We wish to solve for $\tau$.
Recall the formula \eqref{eq:ddetS}. It shows that for the un-extended truncated CMC hierarchy, for which we set $u=0, \bsigma=0$, we have
\be\label{eq:unextau}
\ed(\det\bmcs)=-\tr(\bmcs\bdot{\bphi}) \quad\lra\quad
\tau=\tne^{ \tn{Res}_{\lambda=0}\det\bmcs}.
\ee
Based on this, one may solve for  $\tau$ for the extended CMC hierarchy as follows.
\begin{thm}\label{thm:tau}
For the extended CMC hierarchy, we have
\be\label{eq:extau}
\ed\big(\tn{Res}_{\lambda=0}(\tne^u\det\bmcs)\big)=\varphi_{\bmcs} \quad\lra\quad
\fbox{\quad$\tau=\tne^{\tn{Res}_{\lambda=0}(\tne^u\det\bmcs)}.$\quad} 
\ee
\end{thm}
\begin{proof} 
Eq.\eqref{eq:dS+''} implies (for $\tr(\bmcs)=0$ and $\bmcs^2+\det(\bmcs)I_2=0$),
$$\ed(\bmcs^2)=\bmcs\bdot{\left(\hbphi-\bmcs\bsigma\right)}+\bdot{\left(\hbphi-\bmcs\bsigma\right)}\bmcs.
$$
Scale by $\tne^u$, and from \eqref{eq:u2} one finds
\be\label{eq:euS2}
\ed (\tne^u \bmcs^2)=-\bdot{(\tne^u\bmcs^2\bsigma)}+\tne^u(\bmcs\bdot{\hbphi}+\bdot{\hbphi}\bmcs).
\ee
The LHS of \eqref{eq:extau} follows by
applying the operator $\tcr{-}\tn{Res}_{\lambda=0} \tr$.
\end{proof}
\providecommand{\MR}[1]{}
\providecommand{\bysame}{\leavevmode\hbox to3em{\hrulefill}\thinspace}
\providecommand{\MR}{\relax\ifhmode\unskip\space\fi MR }
\providecommand{\MRhref}[2]{%
  \href{http://www.ams.org/mathscinet-getitem?mr=#1}{#2}
}
\providecommand{\href}[2]{#2}

%--------------------------------------------------------------------------------------------------
%--------------------------------------------------------------------------------------------------
\end{document}